\documentclass[arxiv,reqno,twoside,a4paper,12pt]{amsart}
\usepackage{mltools}
\withdetails
\usepackage{mlmath62,mlthm10-REAR}

\usepackage{upref,amsmath,amssymb,amsthm,mathrsfs}
\usepackage{pgf,tikz, graphics, xypic, latexsym, tensor}
\usepackage{graphicx}
\usepackage[colorlinks=true,linkcolor=blue,citecolor=blue]{hyperref}

\usetikzlibrary{arrows}

\usepackage[scaled]{helvet} % ss
\usepackage{courier} % tt
\usepackage[mathbf]{euler}
\usepackage{mllocal} 
\allowdisplaybreaks
\DeclareMathOperator{\ric}{\mbox{Ric}}

\usepackage{pgfplots}
\numberwithin{equation}{section}

\definecolor{qqwuqq}{rgb}{0,0,0}

\begin{document}

\title[Yamabe flow and bounded geometry]{Normalized Yamabe flow on 
manifolds with bounded geometry}

\author{Bruno Caldeira}
\address{Universidade Federal de S\~{a}o Carlos, Brazil}
\email{brunoccarlotti@gmail.com,brunocarlotti@estudante.ufscar.br}
%\email{}

\author{Luiz Hartmann}
\address{Universidade Federal de S\~{a}o Carlos, Brazil}
\email{hartmann@dm.ufscar.br,luizhartmann@ufscar.br}

\author{Boris Vertman}
\address{Universit\"{a}t Oldenburg, Germany}
\email{vertman@uni-oldenburg.de}

\subjclass[2020]{Primary 53C44 ; Secondary 58J35, 35K08}
\keywords{Yamabe flow, bounded geometry, Schauder estimates}

\thanks{Bruno Caldeira  is partially supported by the Coordenação de Aperfeiçoamento de Pessoal de Nível Superior - Brasil (CAPES) - Finance Code 001 (Process 
	88881.199666/2018-01) and 
	Luiz Hartmann is partially supported by FAPESP:2018/23202-1 and FAPESP:2021/09534-4.}

\maketitle
\begin{abstract}
The goal of this paper is to study Yamabe flow on 
a complete Riemannian manifold of bounded geometry with possibly infinite volume.
In case of infinite volume, standard volume normalization of the Yamabe flow fails and the flow
may not converge. Instead, we consider a curvature normalized Yamabe flow, and 
assuming negative scalar curvature, prove its long-time existence and 
convergence. This extends the results of Su\'arez-Serrato and Tapie to a non-compact setting.
In the appendix we specify our analysis to a particular example of manifolds with bounded geometry, namely manifolds with fibered boundary metric.  
In this case we obtain stronger estimates for the short time solution using microlocal methods.   
\end{abstract}

\tableofcontents

%%%%%%%%%%%%%%%%%%%%%%%%%%%%%%%%%%%%%%%%%%%%%%%%%
\section{Introduction and statement of the main results} \label{introduction}
%%%%%%%%%%%%%%%%%%%%%%%%%%%%%%%%%%%%%%%%%%%%%%%%%%%

The Yamabe flow equation is an evolution equation introduced by Hamilton \cite{hamilton} 
as an alternative ansatz to solving the Yamabe conjecture.
The conjecture, posed by Yamabe \cite{Yamabe} and proved by Trudinger \cite{Trudinger}, Aubin \cite{Aubin} and Schoen \cite{Schoen}, states that for any compact, smooth Riemannian manifold $(M, g)$ without boundary there exists a constant scalar curvature metric, conformal to $g$. Their proofs are based on the calculus of variations and elliptic theory.  
\medskip

Hamilton proposed a new approach using parabolic methods. More precisely, his Yamabe flow is a parabolic evolution equation with solution 
given by a family of Riemannian metrics $\{g(t)\}_{t \in [0,T)}$ on $M$ such that 
\begin{equation}\label{yamabeflow}
    \partial_{t}g(t) = -\scal(g(t))g(t); \;\; g(0) = g.
\end{equation}
On compact Riemannian manifolds the flow shrinks the metric on regions with $\scal(g(t)) > 0$.  In particular a sphere collapses along the flow to a point in finite time. For this reason one introduces a 
volume normalized Yamabe flow, using the average scalar curvature
\begin{equation} \label{rhocompact}
    \rho(t)  = \dfrac{1}{\vol_{g(t)}(M)}\int_{M}\scal(g(t))\dvol(g(t)).
\end{equation}
The volume normalized Yamabe flow is then defined by
\begin{equation}
    \partial_{t}g(t) = (\rho(t)-\scal(g(t)))g(t), \;\; g(0) = g.
\end{equation}

The flow is by now well understood in the setting of compact manifolds. Hamilton \cite{hamilton} himself proved 
long time existence of the volume normalized flow for any choice of initial metric.  
Later, Ye \cite{ye} proved convergence of the flow for scalar negative, scalar flat and locally conformal 
flat scalar positive metrics. The case of metrics that are not conformally flat has been studied 
in a series of papers by Schwetlick and Struwe \cite{SS} and later by Brendle \cite{Brendle, BrendleYF}.
\medskip

In this paper we are concerned with the Yamabe conjecture and more specifically the Yamabe flow on 
non-compact complete manifolds. Note that in case of infinite volume, it is not possible to use \eqref{rhocompact} to construct normalization of the Yamabe flow via average scalar curvature.  Thus, only the unnormalized Yamabe flow has been considered in the literature references below.  \medskip

The Yamabe problem in the non-compact setting has been attacked by elliptic methods by several authors.
Aviles and McOwen \cite{alvesowen} have shown that, under decay assumptions on the scalar curvature and lower bounds for the Ricci curvature, there is a metric in the conformal class with constant positive scalar curvature. Grosse \cite{nadine} has proven that there is a metric with positive constant scalar curvature within the conformal classes of metrics with bounded geometry.  
Wei \cite{wei} studied positive solutions of the Yamabe equations under conditions on the Yamabe invariant and 
volume growth on geodesic balls. 
\medskip

The Yamabe flow in the non-compact setting has been studied on asymptotically conical surfaces by Isenberg, Mazzeo and Sesum 
\cite{IMS}, who proved locally uniform convergence of the Ricci flow -- which equals the Yamabe flow on surfaces -- of a time-rescaled metric to a complete hyperbolic metric with finite area. Ma and Cheng \cite{macheng} have studied the Yamabe flow on complete manifolds assuming a Ricci pinching condition for the initial metric.  Ma, Cheng and Zhu \cite{mcz} have also
studied long-time existence of the Yamabe flow, under some $L^p$ conditions on the scalar curvature.  Within the context of the Yamabe flow on non-compact manifolds, we should also mention the work by Ma \cite{lima19} on conditions for the existence of a metric with constant scalar curvature.  Also in the context of non-compact manifolds, Schulz \cite{schulzpaper} proved global existence of the Yamabe flow with unbounded initial curvature,
provided the metric is conformally equivalent to a complete metric with bounded, non-positive scalar curvature and positive Yamabe invariant. A recent work Ma \cite{lima} establishes global existence of the Yamabe flow on non-compact manifolds that are asymptotically flat near infinity. \medskip

In all of these works, convergence of the flow is out of reach, since \eqref{rhocompact} is not defined and thus
only the unnormalized Yamabe flow has been considered. In this paper we study a different type of normalization 
for the Yamabe flow, that allows to study convergence in the non-compact setting as well.
We use the concepts of decreasing and increasing curvature-normalized flows, denoted by $\mbox{CYF}^{-}$ and $\mbox{CYF}^{+}$ respectively, 
as introduced by Su\'{a}rez-Serrato and Tapie \cite{serratotapie} for compact manifolds
\begin{equation}\label{CYF}
\begin{split}
    \partial_{t}g(t) = (\sup_M \scal(g(t)) -\scal(g(t)))g(t), \;\; g(0) = g, \qquad \left( \mbox{CYF}^{+} \right), \\
     \partial_{t}g(t) = (\inf_M \scal(g(t)) -\scal(g(t)))g(t), \;\; g(0) = g, \qquad \left( \mbox{CYF}^{-} \right).
\end{split}
\end{equation}
We study such curvature 
normalized flows in the setting of manifolds with bounded geometry.

%%%%%%%%%%%%%%%%%%%%%%%%
\subsection{Yamabe flow for the conformal factor}
%%%%%%%%%%%%%%%%%%%%%%%%
The flow preserves the conformal class of the metric and can be written as a scalar evolution equation 
for the conformal factor. More precisely, assume $m:= \dim M \geq 3$ and set $\eta := (m-2)/4$.
Writing $g(t) = u(t)^{1/\eta} g$, the scalar curvature of $g(t)$ can be computed by
($\Delta$ is the negative Laplace Beltrami operator of $(M,g)$)
\begin{equation} \label{scal-trafo}
    \scal(g(t)) = -u(t)^{-(1+1/\eta)}\left[\dfrac{m-1}{\eta}\Delta u(t) - \scal(g) u(t) \right].
\end{equation} 
In view of this relation, the Yamabe flow \eqref{yamabeflow} turns into
\begin{equation} \label{yf-u}
\begin{split}
    &\partial_{t} u(t)^{(m+2)/(m-2)} = \dfrac{m+2}{m-2}\Bigl( (m-1)\Delta u(t) - \eta\scal(g)u(t)\Bigr)
   \\ \Leftrightarrow  \quad &\partial_{t}u(t) = (m-1)u(t)^{-1/\eta}\Delta u(t) - \eta\scal(g)u(t)^{1-1/\eta},
    \end{split}
\end{equation}
with the initial condition $u(t=0)=1$.

%%%%%%%%%%%%%%%%%%%%%%%%
\subsection{Normalized Yamabe flows for the conformal factor}
%%%%%%%%%%%%%%%%%%%%%%%%
Similar computations as those leading to \eqref{yf-u}
yield the following scalar evolution equation for the conformal factor under the
volume normalized Yamabe flow
  \begin{equation} 
        \partial_{t}u(t) - (m-1)u(t)^{-1/\eta}\Delta u(t) = \eta \Bigl(\rho(t) u(t) - \scal(g)u(t)^{1-1/\eta}\Bigr).
   \end{equation}
The curvature normalized flows in \eqref{CYF} are similarly given by
  \begin{equation}\label{CYF-u}
  \begin{split}
        \partial_{t}u(t) &- (m-1)u(t)^{-1/\eta}\Delta u(t) \\ 
        &= \eta \Bigl(\sup_M \scal(g(t)) \cdot u(t) - \scal(g)u(t)^{1-1/\eta}\Bigr), \qquad \left( \mbox{CYF}^{+} \right), \\
        \partial_{t}u(t) &- (m-1)u(t)^{-1/\eta}\Delta u(t) \\ 
        &= \eta \Bigl(\inf_M \scal(g(t)) \cdot u(t) - \scal(g)u(t)^{1-1/\eta}\Bigr), \qquad \left( \mbox{CYF}^{-} \right).
            \end{split}
   \end{equation}
%%%%%%%%%%%%%%%%%%%%%%%%
\subsection{Outline of the paper and main results}
%%%%%%%%%%%%%%%%%%%%%%%%

In \S \ref{bdedgeo}, we review the geometry of manifolds with bounded geometry and introduce 
the corresponding family of H\"older spaces $C^{k,\alpha}(M)$. In \S \ref{omori} we employ the Omori-Yau maximum
principle to establish the uniqueness of solutions to the Yamabe flow within $C^{2,\alpha}(M)$ and derive differential inequalities
for the solution, that will later be used for the a priori estimates. \medskip

In \S \ref{parabolic-estimates} we study some parabolic Schauder estimates for the inhomogeneous heat-type equation 
$(\partial_t - a\cdot \Delta)u = f$.  Based on such mapping properties, in \S
\ref{shorttime} we establish in Theorem \ref{theoremshorttime} the short time existence of the (unnormalized) 
Yamabe flow \eqref{yamabeflow} within the class of 
manifolds with bounded geometry. \medskip

In \S \ref{renormalization}, we turn to the increasing curvature normalized Yamabe flow ($\mbox{CYF}^{+}$), introduced in \eqref{CYF-u}, whose short-time existence follows from Theorem \ref{theoremshorttime} by some time rescaling. The same holds also for the decreasing curvature normalized Yamabe flow 
$\mbox{CYF}^{-}$ by a verbatim repetition of the arguments and hence we only write the proofs for $\mbox{CYF}^{+}$. \medskip

In \S \ref{evolutionsection}, we study the evolution of $\scal(g)$ along the $\mbox{CYF}^{+}$. In \S \ref{uniformsection} we derive a priori estimates for solutions of the increasing curvature normalized Yamabe flow. These a priori estimates allow us
to apply the machinery obtained in \S \ref{parabolic-estimates} to conclude
the global existence of $\mbox{CYF}^{+}$ on manifolds with bounded geometry in \S \ref{longtimesection}. This is 
our first main result that we state here. 

\begin{thm} \label{longtime}
Let $(M,g_{0})$ be a manifold with bounded geometry of dimension $m\geq 3$ with negative scalar curvature $\scal(g_{0}) \in C^{k,\alpha}(M)$, 
uniformly bounded away from zero and $k\geq 4$. Then the increasing (or decreasing) curvature normalized Yamabe flow $\mbox{CYF}^{\pm}$ (see \Eqref{CYF-u}) 
admits a global solution $g = u^{4/(m-2)}g_{0}$ for some $u \in C^{k,\alpha}(M\times \mathbb{R}_{+})$.
\end{thm}

Finally, in \S \ref{section-convergence} we establish the convergence for the $\mbox{CYF}^{\pm}$,
which proves the Yamabe conjecture on negatively curved manifolds with bounded geometry.
Our result, see Theorem \ref{convergencelong} for the precise statement, reads as follows.

\begin{thm}\label{convergence}
Let $(M,g_{0})$ be a manifold with bounded geometry of dimension $m\geq 3$ 
such that $\scal(g_{0}) \in C^{4,\alpha}(M)$ is negative and 
uniformly bounded away from zero. Then the increasing (or decreasing) curvature normalized Yamabe flow $\mbox{CYF}^{\pm}$
converges to a Riemannian metric $g^{*}$ conformal to $g_{0}$ 
with constant negative scalar curvature.
\end{thm}

One can view our 
contribution as an extension of Su\'{a}rez-Serrato and Tapie \cite{serratotapie} to a non-compact setting.  Let us conclude with a remark, that even though we only write out the proofs for $\mbox{CYF}^{+}$,
same statements hold for the decreasing curvature normalized Yamabe flow 
$\mbox{CYF}^{-}$ as well.

\subsection*{Appendix on $\Phi$-manifolds}

We dedicate a section on a special class of non-compact manifolds with infinite volume: the class of $\Phi$-manifolds.  These are manifolds with fibered boundary equipped with a particular Riemannian metric in its open interior and are, as explained in the appendix, manifolds with bounded geometry.  Thus, every result presented in this work holds, in particular, for $\Phi$-manifolds.  However, due to its special structure at the boundary, one can actually obtain more refined estimates of solutions, using arguments similar to the ones in \cite{bahuaud2014yamabe}. This allows us to prove short-time existence of the Yamabe flow in weighted H\"older spaces, which is not possible in the general case.

\subsubsection*{Acknowledgments} 
Bruno Caldeira thanks the Univers\"{a}t Oldenburg for the hospitality.
%and to the 
%Co\-or\-de\-na\-\c{c}\~{a}o de Aperfei\c{c}oamento de Pessoal de N\'{i}vel 
%Superior (CAPES-Brasil- Finance Code 001) for the financial support 
%(Process 
%88881.199666/2018-01). 
Boris Vertman thanks Tobias Marxen and also Gilles Carron (Universit\'e 
Nantes) for useful discussions and specifically for pointing out the paper 
by Su\'{a}rez-Serrato and Tapie \cite{serratotapie}.

%%%%%%%%%%%%%%%%%%%%%%%%%%%%%%%%%%%%%%%%%%
\section{Manifolds of bounded geometry and H\"older spaces} \label{bdedgeo}
%%%%%%%%%%%%%%%%%%%%%%%%%%%%%%%%%%%%%%%%%%

This section reviews the classical concept of manifolds of bounded geometry, as well as some of its basic consequences. 

\begin{defin}
A Riemannian manifold $(M,g_0)$ is said to have bounded geometry if it satisfies two conditions:
\begin{enumerate}
    \item its injectivity radius  is bounded uniformly from below away from zero, 
    i.e. there exists some uniform constant $\delta > 0$, such that for any open metric ball $B_{\delta}(p) \subset M$ 
    of radius $\delta$, centered at $p$, the exponential map
    \begin{equation*}
        \exp_{p}:B_{\delta}(0)\subset T_{p}M\rightarrow B_{\delta}(p)\subset M
    \end{equation*}
    is a diffeomorphism;
    \item its Ricci curvature is uniformly bounded, i.e. there exists an uniform constant $c'>0$ such that for any vector field $V$, one has
\begin{equation*}
    |\ric(V,V)| \leq c'g_{0}(V,V).
\end{equation*}
\end{enumerate} 
\end{defin}

By definition, a manifold $M$ of bounded geometry admits a cover by metric open balls $\{B_{\delta}(p)\}_{p \in M}$,  of uniform radii $\delta > 0$, that are quasi-isometric to the Euclidean open ball $B_{\delta}(0) \in \mathbb{R}^{m}$, where $m = \dim M$.  This means that for each point $p \in M$ the exponential map defines a diffeomorphism
\begin{equation}
    \Psi'_{p}: B_{\delta}(0) \rightarrow B_{\delta}(p)
\end{equation}
which re-scales the distances between points in each ball by a constant factor which is independent on the choice of the point $p$.  
We also consider the corresponding diffeomorphism between parabolic balls
\begin{equation}\label{Psi-pullback}
    \Psi_{p}: Q_{\delta}(0) := B_{\delta}(0)\times [0,T] \rightarrow B_{\delta}(p)\times [0,T].
\end{equation}
Given a manifold with bounded geometry $(M,g_0)$, we can define the corresponding H\"older spaces as follows.
Consider the classical H\"older space $C^{k,\alpha}(Q_\delta)$ with 
H\"older norm denoted by $\| \cdot \|_{k,\alpha, Q_\delta}$. Then the H\"older norm $\| \cdot \|_{k,\alpha}$
on $C^{k,\alpha}(M \times [0,T])$, is defined by
\begin{align}\label{local-global-norm}
\sup_{p \in M} \| \Psi_{p}^* u \|_{k,\alpha, Q_\delta}.
\end{align}

In fact, we will now observe that such local H\"older regularity in open metric balls $\{B_{\delta}(p)\}_{p \in M}$ is equivalent to a 
seemingly global H\"older regularity, i.e. we can alternatively define the H\"older spaces as follows. 

\begin{defin} \label{holder}
The H\"{o}lder space $C^{\alpha}(M\times [0,T])$, for $\alpha \in (0,1)$, is defined as the space of continuous functions 
$u \in C^{0}(M\times [0,T])$ which satisfy
\begin{equation}
    [u]_{\alpha} := \sup_{M^2_T} \left\{ \dfrac{|u(p,t)-u(p',t')|}{d(p,p')^{\alpha}+|t-t'|^{\alpha/2}}\right\} < \infty,
\end{equation}
where the supremum is taken over $M^2_T$ with $M_T:=M \times [0,T]$ and $d$ is the Riemannian distance function induced by $g_{0}$.  
The H\"{o}lder norm of any $u \in C^{\alpha}(M\times [0,T])$ is defined by
\begin{equation}
    \|u\|_{\alpha} := \|u\|_{\infty} + [u]_{\alpha}.
\end{equation}
\end{defin}

We also define the higher order H\"{o}lder spaces for any given $k \in \mathbb{N}$ 
as a subset of $k$ times continuously differentiable functions $C^k$
\begin{equation*}
  C^{k,\alpha}(M\times [0,T]) = 
    \biggsetdef{ u \in C^{k}(M\times [0,T])} 
        {\begin{array}{l}  (V \circ \partial^{l_{2}}_{t})u \in C^{\alpha}(M\times [0,T]),
           \\ \mbox{for} \; V \in {}^b\mbox{Diff}^{l_1}(M), \hspace{2mm} l_{1} + 2l_{2} \leq k 
        \end{array}}
\end{equation*}
where ${}^b \textup{Diff}^{l_1}(M)$ consist of those differential operators $D$ of $l_1$-th order, 
where the coefficients of $(\Psi'_p)^* \circ D \circ ((\Psi'_p)^{-1})^*$ are bounded
and $C^\alpha$ uniformly for all $p \in M$. This is a Banach space with the norm
\begin{equation}\label{global-norm}
\|u\|_{k,\alpha} := \|u\|_{\alpha} +\displaystyle\sum_{l_{1} + 2l_{2} \leq k} 
\sum_{V \in {}^b \textup{Diff}^{l_1}(M)} \|(V\circ \partial_{t}^{l_{2}})u\|_{\alpha}.
\end{equation}

\noindent The resulting normed vector spaces $C^{k,\alpha}(M\times [0,T])$ are Banach spaces.
Equivalence of \eqref{local-global-norm} and \eqref{global-norm} follows from a
simple observation.

\begin{lem}\label{local-norm}
The following defines an equivalent norm on $C^{\alpha}(M\times [0,T])$
\begin{equation}\label{eq.local-norm}
    \|u\|'_{\alpha} := \|u\|_{\infty} + [u]'_{\alpha}, \quad 
     [u]'_{\alpha} := \sup_{M^2_{T,\delta}} \left\{ \dfrac{|u(p,t)-u(p',t')|}{d(p,p')^{\alpha}+|t-t'|^{\alpha/2}}\right\},
\end{equation}
where the supremum is taken over 
$$M^2_{T,\delta} := \{(p,t), (p',t') \in M_T \mid d(p,p')^{\alpha}+|t-t'|^{\alpha/2} \leq \delta\}.$$
More precisely, we have the following relation between the two norms
$$\|u\|'_{\alpha} \leq  \|u\|_{\alpha} \leq (1+2\delta^{-1}) \|u\|'_\alpha.$$
\end{lem}

\begin{proof}
It is clear that $ \|u\|'_{\alpha} \leq  \|u\|_{\alpha}$. To prove the second
estimate, simply note for any $u \in C^{\alpha}(M\times [0,T])$
and any $(p,t), (p',t') \in M_T$ with $$d(p,p')^{\alpha}+|t-t'|^{\alpha/2}\geq \delta,$$
that we can estimate the H\"older differences as follows
\begin{align*}
   \dfrac{|u(p,t)-u(p',t')|}{d(p,p')^{\alpha}+|t-t'|^{\alpha/2}} \leq  \dfrac{|u(p,t)-u(p',t')|}{\delta}
   \leq 2 \delta^{-1} \| u \|_\infty.
\end{align*}
\end{proof}

As a consequence, we conclude that the various H\"older norms are equivalent,  
an observation that will be convenient in the appendix. 

\begin{cor}\label{local-global} The H\"older norms  \eqref{local-global-norm} and \eqref{global-norm}
are equivalent. \end{cor}

\begin{proof}
The statement follows from \eqref{Psi-pullback} and the fact that, for the uniform radii $\delta$, 
the distance function on $(M,g_0)$ and the Euclidean distance are uniformly equivalent.
\end{proof}

\begin{rmk}
Sometimes, we will also use H\"{o}lder spaces for functions depending either only on 
spacial variables or on time variables, denoted as $C^{k,\alpha}(M)$ and $C^{k,\alpha}([0,T])$, with H\"{o}lder brackets (for $k=0$)
\begin{align*}
   [u]_{\alpha} = \sup \dfrac{|u(p) - u(p')|}{d(p,p')^{\alpha}} \;\; \mbox{and} \;\; [u]_{\alpha} = \sup \dfrac{|u(t)-u(t')|}{|t-t'|^{\alpha/2}},
\end{align*}
respectively.
\end{rmk}

%%%%%%%%%%%%%%%%%%%%%%%%%%%%%%%%%%%%%%%%%%
\section{Omori-Yau maximum principle and uniqueness of solutions} \label{omori}
%%%%%%%%%%%%%%%%%%%%%%%%%%%%%%%%%%%%%%%%%%

In this section, we will prove some consequences of the Omori-Yau maximum principle on manifolds with bounded geometry.
Given a manifold $(M,g_{0})$ of bounded geometry, and the (negative) Laplace Beltrami operator $\Delta_{g_{0}}$, associated to $g_{0}$,
the Omori-Yau maximum principle (for the Laplacian) asserts that for any function $u \in C^{2}(M)$ there is a sequence $\{p_{k}\}_{k} \subset M$ 
satisfying the following estimates:
\begin{equation}\label{omori-sup}
    u(p_{k}) > \displaystyle\sup_{M}u - \dfrac{1}{k} \;\; \mbox{and} \;\; \Delta_{g_{0}}u(p_{k}) < \dfrac{1}{k}.
\end{equation}
A similar statement holds for functions $u \in C^{2}(M)$ near its infimum values, which means that there 
exists a sequence $\{p'_{k}\}_{k} \subset M$ such that
\begin{equation} \label{omori-inf}
    u(p'_{k}) < \inf_{M} u + \dfrac{1}{k} \;\; \mbox{and} \;\; \Delta_{g_{0}}(p'_{k}) > \dfrac{1}{k}.
\end{equation}
By \cite[Theorem 2.3]{aliasrigoli}, the Omori-Yau principle for the Laplacian holds on any manifold
with Ricci curvature bounded from below, and in particular on manifolds with bounded geometry. 

%%%%%%%%%%%%%%%%%%%%%%%%%%%%%%%%%%%%%%%%%%
\subsection{Some enveloping theorem} 
%%%%%%%%%%%%%%%%%%%%%%%%%%%%%%%%%%%%%%%%%%

\begin{prop}\label{envelope}
Consider any $u \in C^{2,\alpha}(M\times [0,T])$. Then the functions
$$u_{\sup} (t) := \sup_{M}u(\cdot, t), \quad u_{\inf}(t):= \inf_{M}u(\cdot, t)$$ are differentiable
almost everywhere in $(0,T)$ and at those $t \in (0,T)$ we find, in the notation of 
\eqref{omori-sup} and \eqref{omori-inf},
\begin{equation}\label{dini}
\begin{split}
&\frac{\partial}{\partial t}u_{\sup} (t) \leq \lim_{\epsilon \to 0} 
\left( \limsup_{k\to \infty} \frac{\partial u}{\partial t} \left(p_k(t+\epsilon), t + \epsilon\right)\right), \\
&\frac{\partial}{\partial t} u_{\inf} (t) \geq \lim_{\epsilon \to 0} 
\left( \liminf_{k\to \infty} \frac{\partial u}{\partial t} \left(p'_k(t+\epsilon), t + \epsilon\right)\right).  
\end{split}
\end{equation}
\end{prop}

\begin{proof}
Apply \eqref{omori-sup} to $u(t+\epsilon)$ and find by the mean value theorem 
\begin{align*}
u_{\sup} (t+\epsilon) &\leq u(p_k(t+\epsilon), t + \epsilon) + \frac{1}{k}
\\ &= u(p_k(t+\epsilon), t) + \epsilon \cdot \frac{\partial u}{\partial t}(p_k(t+\epsilon), \xi) + \frac{1}{k},
\end{align*}
for some $\xi \in (t,t+ \epsilon)$. On the other hand, we can write
\begin{align*}
u_{\sup} (t+\epsilon) &= u_{\sup} (t) + \epsilon \cdot 
\frac{u_{\sup} (t+\epsilon) - u_{\sup} (t)}{\epsilon} \\
&\geq u(p_k(t+\epsilon),t) + \epsilon \cdot 
\frac{u_{\sup} (t+\epsilon) - u_{\sup} (t)}{\epsilon}. 
\end{align*}
Combining these two estimates leads, after cancelling $u(p_k(t+\epsilon),t)$, to 
\begin{align*}
 \epsilon \cdot \frac{u_{\sup} (t+\epsilon) - u_{\sup} (t)}{\epsilon}
\leq \epsilon \cdot \frac{\partial u}{\partial t}(p_k(t+\epsilon), \xi) + \frac{1}{k}.
\end{align*}
Taking the limit superior as $k \to \infty$ on the right hand side, we obtain
\begin{align*}
 \epsilon \cdot \frac{u_{\sup} (t+\epsilon) - u_{\sup} (t)}{\epsilon}
\leq \epsilon \cdot \limsup_{k\to \infty}\frac{\partial u}{\partial t}(p_k(t+\epsilon), \xi).
\end{align*}
Canceling $\epsilon$ on both sides, we find
\begin{equation}\label{est1}
\begin{split}
&\frac{u_{\sup} (t+\epsilon) - u_{\sup} (t)}{\epsilon}
\leq \limsup_{k\to \infty}\frac{\partial u}{\partial t}(p_k(t+\epsilon), \xi)
\\ &= \limsup_{k\to \infty} \left( \frac{\partial u}{\partial t}(p_k(t+\epsilon), \xi)
- \frac{\partial u}{\partial t}(p_k(t+\epsilon), t+\epsilon) \right)
\\ &+ \limsup_{k\to \infty}\frac{\partial u}{\partial t}(p_k(t+\epsilon), t+\epsilon).
\end{split}
\end{equation}
For any $u \in C^{2,\alpha}(M\times [0,T])$ we can estimate
\begin{equation}\label{est2}
\begin{split}
&\bullet \quad  \limsup_{k\to \infty} \left| \frac{\partial u}{\partial t}(p_k(t+\epsilon), \xi)
- \frac{\partial u}{\partial t}(p_k(t+\epsilon), t+\epsilon) \right| \leq \|u\|_{2,\alpha} \epsilon^{\alpha/2}, \\
&\bullet \quad \limsup_{k\to \infty}\left| \frac{\partial u}{\partial t}(p_k(t+\epsilon), t+\epsilon) \right| \leq \|u\|_{2,\alpha}.
\end{split}
\end{equation}

Thus the last two summands in 
\eqref{est1} are bounded uniformly in $\epsilon$. Repeating the same arguments with 
the roles of $u(t)$ replaced by $u(t+\epsilon)$ interchanged, we conclude that
$u_{\sup}$ is locally Lipschitz and thus by the theorem of Rademacher, differentiable
almost everywhere. This proves the first statement. 
\medskip

At those $t\in (0,T)$, where $u_{\sup}$ is differentiable, we conclude
from \eqref{est1} and the first line in \eqref{est2}, taking $\epsilon \to 0$
\begin{equation}
\begin{split}
\frac{\partial}{\partial t} u_{\sup}(t) \leq \lim_{\epsilon \to 0} 
\left( \limsup_{k\to \infty} \frac{\partial u}{\partial t} \left(p_k(t+\epsilon), t + \epsilon\right)\right).  
\end{split}
\end{equation}
This proves the first inequality in \eqref{dini}. The second inequality
follows from the first, using \eqref{omori-inf}, with $u$ replaced by $(-u)$.
\end{proof}

%%%%%%%%%%%%%%%%%%%%%%%%%%%%%%%%%%%%%%%%%%
\subsection{Uniqueness of solutions} 
%%%%%%%%%%%%%%%%%%%%%%%%%%%%%%%%%%%%%%%%%%

We can now turn to uniqueness of solutions to the Yamabe flow. 
In view of Proposition \ref{envelope}, we can now prove the following result. 

\begin{prop}\label{maxprinccor}
Let $a$ be a bounded positive function and $b$ be a bounded non-negative function on $M \times [0,T]$.
Let $u \in C^{2,\alpha}(M\times [0,T])$ be a solution to 
$$
\partial_t  u = a \Delta_{g_{0}} u -b u, 
$$
with initial value $0$ at $t=0$. Then $u\equiv 0$.
\end{prop}

\begin{proof}
Consider first the case where $b=0$. Note first by \eqref{omori-sup} and \eqref{omori-inf}
\begin{align*}
\frac{\partial}{\partial t} u\bigl(p_k(t),t\bigr) \leq \frac{a\bigl(p_k(t),t\bigr)}{k}, \quad
\frac{\partial}{\partial t} u\bigl(p'_k(t),t\bigr) \geq -\frac{a\bigl(p'_k(t),t\bigr)}{k}.
\end{align*}
Then in view of Proposition \ref{envelope} we find almost everywhere
$$
\frac{\partial}{\partial t} u_{\sup}(t) \leq 0, \quad \frac{\partial}{\partial t} u_{\inf}(t) \geq 0.
$$
Then in view of $u(t=0)=0$, we conclude $u\equiv 0$. Now the general statement follows
as in \cite[Corollary 9.2]{giuseppebruno}. 
\end{proof}

\begin{cor} \label{uniqueness}
	Consider the Yamabe flow equation as in \Eqref{yf-u}
	\begin{equation}\label{yf-u2}
	\partial_{t}u = (m-1)u^{-1/\eta}\Delta_{g_{0}}u - \eta \scal(g_{0})u^{1-1/\eta}, \;\; u|_{t=0} = u_{0},
	\end{equation}
	for some positive initial data $u_{0}\in C^{2,\alpha}(M)$.  
	For such a Cauchy problem, a positive solution in $C^{2,\alpha}(M\times [0,T])$ 
	is unique for any given $0<T<\infty$.
\end{cor}

\begin{proof}
	Suppose $u$ and $v$ are two positive solutions in $C^{2,\alpha}(M)(M\times [0,T])$ for 
	\eqref{yf-u2}. Consider $\omega = u - v \in C^{2,\alpha}(M)(M\times [0,T])$. Since 
	$u(t=0)=v(t=0) =u_0$, we find $\omega(t=0) = 0$. Moreover, we infer from \eqref{yf-u2}
	\begin{equation*}
	u^{1/\eta}\partial_{t}u - v^{1/\eta}\partial_{t}v = (m-1)\Delta_{g_{0}}\omega - \eta\scal(g_{0})\omega.
	\end{equation*}
	From the definition of $\omega$, we have
	\begin{align*}
	\partial_{t}\omega = &u^{-1/\eta}\left(u^{1/\eta}\partial_{t}u - v^{1/\eta}\partial_{t}v + (v^{1/\eta} - u^{1/\eta})\partial_{t}v\right) \\
	= &u^{-1/\eta}\left((m-1)\Delta_{g_{0}}\omega - \eta\scal(g_{0})\omega + (v^{1/\eta} - u^{1/\eta})\partial_{t}v\right) \\
	= &- \left(\eta\scal(g_{0})u^{-1/\eta} + \dfrac{\partial_{t}v}{\eta}\int_{0}^{1}(sv + (1-s)u)^{1/\eta - 1}\di s\right)\omega
	\\ &\quad \qquad \qquad \qquad \qquad \qquad \qquad \qquad +(m-1)u^{-1/\eta}\Delta_{g_{0}}\omega,
	\end{align*}
	where the last equality follows from the Taylor's theorem applied for the function 
	$f(s):= (sv + (1-s)u)^{1/\eta}$.  This means that $\omega$ is a solution of the equation 
	\begin{equation*}
	\partial_{t}\omega = a\Delta_{g_{0}}\omega + b\omega,
	\end{equation*}
	with $a \in C^{2,\alpha}(M\times [0,T])$ positive and $b \in C^{\alpha}(M\times [0,T])$.  Since nothing can be said about the sign of the $b$-term above, we consider any negative constant 
$c<-\|b\|_{\infty}$ and apply an integration factor trick by writing $\omega' = e^{ct}\omega$.  
We obtain an equation for $\omega$
	\begin{equation*}
	\partial_{t}\omega' = a\Delta_{g_{0}}\omega' + (b+c)\omega',
	\end{equation*}
	with $\omega'|_{t=0} = \omega|_{t=0} = 0$.  Now, since $c<-\|b\|_{\infty}$, we have $(b+c) < 0$.  
	From Proposition \ref{maxprinccor}, it follows that $\omega' \equiv 0$ and, consequently, $\omega \equiv 0$.
\end{proof}

%%%%%%%%%%%%%%%%%%%%%%%%%%%%%%%%%%%%%%%%%%
\subsection{Some differential inequalities for solutions to $\mbox{CYF}^{+}$} 
%%%%%%%%%%%%%%%%%%%%%%%%%%%%%%%%%%%%%%%%%%

As a direct consequence of Proposition \ref{envelope} we also obtain differential inequalities for 
solutions to the increasing curvature normalized Yamabe flow $\mbox{CYF}^{+}$. These 
will be central later in the derivation of a priori estimates. 

\begin{cor}\label{diff-ineq}
Let $u \in C^{2,\alpha}(M\times [0,T])$ be a positive (uniformly bounded away from zero) solution to 
the increasing curvature normalized Yamabe flow $\mbox{CYF}^{+}$ in \eqref{CYF-u}.
Then almost everywhere in $(0,T)$
  \begin{equation}
  \begin{split}
       &\frac{\partial}{\partial t} u_{\sup}
       \leq \eta \sup_M \scal(g(t)) \cdot u_{\sup} + \eta \sup_M |\scal(g_{0})| \cdot u_{\sup}^{1-1/\eta}, \\
        &\frac{\partial}{\partial t} u_{\inf}
       \geq \eta \sup_M \scal(g(t)) \cdot u_{\inf} + \eta \inf_M |\scal(g_{0})| \cdot u_{\inf}^{1-1/\eta}.
            \end{split}
   \end{equation}
\end{cor}

\begin{proof}
Note first by \eqref{CYF-u} and \eqref{omori-sup}
  \begin{equation}
  \begin{split}
        \frac{\partial}{\partial t} u\bigl(p_k(t),t\bigr) &\leq \frac{(m-1)}{k} \cdot u^{-1/\eta}\bigl(p_k(t),t\bigr) 
        + \eta \sup_M \scal(g(t)) \cdot u\bigl(p_k(t),t\bigr) \\ & - \, \eta \scal(g_{0})\bigl(p_k(t),t\bigr) \cdot u\bigl(p_k(t),t\bigr)^{1-1/\eta}.
   \end{split}
   \end{equation}
Since $u$ is positive and uniformly bounded away from zero, we conclude
  \begin{equation}
  \begin{split}
        \limsup_{k\to \infty} \frac{\partial u}{\partial t} \bigl(p_k(t),t\bigr)
        &\leq \eta \sup_M \scal(g(t)) \cdot u_{\sup}(t) \\ &+ \eta \sup_M |\scal(g_{0})| \cdot u_{\sup}(t)^{1-1/\eta}.
   \end{split}
   \end{equation}
Now the first statement follows from Proposition \ref{envelope}.
The second statement follows by \eqref{omori-inf} along the same lines.
\end{proof}

%%%%%%%%%%%%%%%%%%%%%%%%%%%%%%%%%%%%%%%%%%%%%%%%%%%%%
\section{Parabolic Schauder estimates in bounded geometry} \label{parabolic-estimates}
%%%%%%%%%%%%%%%%%%%%%%%%%%%%%%%%%%%%%%%%%%%%%%%%%%%%%%

In order to prove short-time existence of the Yamabe flow on manifolds with bounded geometry, 
one must first obtain mapping properties for a parametrix for appropriate version of the heat equation.
In a previous work \cite{bahuaud2019long}, a priori estimates like in Theorem \ref{uniformboundu} and 
Proposition \ref{uniformboundpartialu} were converted into H\"older regularity by mapping properties of some
heat parametrix for $(\partial_t + u^{-1/\eta} \Delta_{g_{0}})$. This approach fails here since the parametrix
construction in \cite[\S 4]{bahuaud2019long} does not work for $u \in C^{k,\alpha}(M\times [0,T])$ in a 
general setting of bounded geometry. Therefore, we argue here by reducing to classical parabolic Schauder 
estimates.  To achieve this, we use the classical Krylov-Safonov estimate, see \cite{KS} and the nice exposition in 
\cite[Theorem 12]{picard}.  For convenience of the reader, they are stated below:
\begin{thm}\cite[Theorem 12]{picard} \label{krylov}
Let $u \in C^{k+2}(Q_{2\delta})$, $k\geq 2$, such that $u$ satisfies
\begin{equation*}
    \left(\partial_{t} - \displaystyle\sum_{i,j}a^{ij}\partial_{x_{i}}\partial_{x_{j}}\right)u = f,
\end{equation*}
with $a^{ij} = a^{ji}$ and $\Lambda^{-1}\delta^{ij}\leq a^{ij} \leq \Lambda\delta^{ij}$ for some $\Lambda>0$.  Then $f \in L^{\infty}(Q_{2\delta})$ implies  \begin{equation*}
    \|u\|_{\alpha;Q_{\delta}} \leq C\left(\|u\|_{\infty;Q_{2\delta}} + \|f\|_{\infty;Q_{2\delta}}\right).
\end{equation*}
\end{thm}

\begin{prop}\label{Q-mapping-properties}
Consider $a \in C^{k,\alpha}(M)$ positive, uniformly bounded away from zero. 
Then the in-homogeneous heat equation 
\begin{equation} \label{cp1}
    (\partial_t - a \cdot \Delta_{g_{0}}) u = f, \;\; u(t=0)=0,
\end{equation} 
with $f \in C^{k,\alpha}(M\times [0,T])$.  Then
\begin{enumerate}
    \item There exists a constant $C>0$ independent of $T$ such that
    \begin{equation}
        \|u\|_{k+2,\alpha} \leq C\ \Big(\|u\|_{\infty} + \|f\|_{k,\alpha}\Bigr);
    \end{equation}
    \item \Eqref{cp1} has a parametrix $Q$ acting as a bounded linear map
    \begin{equation}\label{Q-parametrix-mapping}\begin{split}
&Q:C^{k,\alpha}(M\times [0,T]) \rightarrow C^{k+2,\alpha}(M\times [0,T]), \\
&Q:C^{k+2,\alpha}(M\times [0,T]) \rightarrow t \ C^{k+2,\alpha}(M\times [0,T]).
\end{split}
\end{equation}
\end{enumerate}
\end{prop}

\begin{proof}
We will proceed by reducing the argument to local $\delta$-balls.  To do so, let us consider the quasi-isometries as in \eqref{Psi-pullback}.  In the proof that follows, we omit the subscript in $\Psi$ for simplicity.  
First, assume the function $u$ to satisfy \eqref{cp1}.  Thus, the following equation holds:
\begin{equation*}
    \Bigl(\partial_{t} - \Psi^{*}a\cdot \widetilde{\Delta_{g_{0}}}\Bigr)\Psi^{*}u = \Psi^{*}f.
\end{equation*}
where $\widetilde{\Delta_{g_{0}}}$ is the pullback of $\Delta_{g_{0}}$ via $\Psi$.  Set $Q_{\delta} := B_{\delta}(0)\times [0,\delta^{2}]$.  By the Krylov-Safonov estimate, 
see \cite{KS} and cf. \cite[Theorem 12]{picard}, we find for some uniform constant $C>0$, 
depending only on $\delta, \|u\|_\infty$ and $\|a\|_\infty$
\begin{align*}
\| \Psi^* u \|_{\alpha, Q_{\delta/2}} &\leq C \Bigl(\| \Psi^* u \|_{\infty,Q_\delta}  + \| \Psi^* f \|_{\infty,Q_\delta} \Bigr) \\
&\leq C \Bigl(\| u \|_{\infty}  + \| f \|_{\infty} \Bigr).
\end{align*}
Thus $\Psi^* u \in C^{\alpha}(Q_{\delta/2})$. By Lemma \ref{local-global} we conclude $u \in C^{\alpha}(M\times [0,\delta^{2}/4])$.
We extend the regularity statement to the whole time interval $[0,T]$ (with constants independent of $T$) 
iteratively, by setting $t= \delta^2 + t'$ and obtaining by the argument above $u \in C^{\alpha}(M\times [\delta^{2}/4, \delta^{2}/2])$,
and repeating the iteration, until we reach $T$.  Now, by proceeding similarly as above by using the classical parabolic Schauder estimates, see \cite{krylov} and cf. 
\cite[Theorem 6]{picard}, we get
\begin{equation}\label{parabolic-Schauder}
\begin{split}
\| \Psi^* u \|_{k+2, \alpha, Q_\delta/2} &\leq C \Bigl(\| \Psi^* u \|_{\infty,Q_\delta}  + \| \Psi^* f \|_{k, \alpha, Q_\delta} \Bigr)
\\ &\leq C \Bigl(\| u \|_{\infty}  + \| f \|_{k,\alpha} \Bigr).
\end{split}\end{equation}
Repeat the argument as presented above to extend the estimate to the entire interval $[0,T]$.  Thus, $u \in C^{k+2,\alpha}(M\times [0,T])$ and the first property is verified.

Now, for the second item, consider the in-homogeneous heat equation with $f \in C^{k,\alpha}(M\times [0,T])$ and 
initial value $u_0 \in C^{k+2,\alpha}(M)$
$$
(\partial_t - a \cdot \Delta_{g_{0}}) u = f, \;\;  u(t=0)=u_0.
$$

  Then, by reducing the argument to local $\delta$-balls as in \S \ref{parabolic-estimates}, 
we can follow the proof of \cite[Theorem 5.1 on p.320]{Lad}, and conclude
for some uniform constant $C>0$ existence of a unique solution $u \in C^{k+2,\alpha}(M\times [0,T])$ with
$$
\| u \|_{k+2,\alpha} \leq C \Bigl( \|f\|_{k,\alpha} + \|u_0\|_{k+2, \alpha}\Bigr). 
$$
This proves the first mapping property in \eqref{Q-parametrix-mapping} by setting $u_0=0$.
For the second mapping property in \eqref{Q-parametrix-mapping}, set $f=0$ and obtain 
a solution $u= Ru_0$ with the solution operator $R$ acting as a bounded
linear map 
$$R: C^{k+2,\alpha}(M) \to C^{k+2,\alpha}(M \times [0,T]).$$
The solution operator $Q$ of the in-homogeneous problem is then given by
\begin{equation}
Qf(p,t) = \int_0^t \Bigl( Rf(\tilde{t}) \Bigr) (p,t-\tilde{t}) d\tilde{t}.
\end{equation}
Indeed, a direct computation shows 
\begin{align*}
\Bigl(\partial_t - &a \cdot \Delta_\Phi\Bigr) Qf(p,t) \\
= \, &f(p,t) + \int_0^t \Bigl( \partial_t - a \cdot \Delta_\Phi \Bigr) \Bigl( Rf(\tilde{t}) \Bigr) (p,t-\tilde{t}) d\tilde{t} \\
= \, &f(p,t).
\end{align*}
This implies directly the second mapping property in \eqref{Q-parametrix-mapping} and 
completes the proof.
\end{proof}

\begin{rmk}
The mapping properties obtained in the previous result hold also 
for the parametrix $H$ of the heat equation $(\partial - \Delta_{g_{0}})$, since the constant function equal to $1$ satisfies the hypothesis in Proposition \ref{Q-parametrix-mapping}.
\end{rmk}

%%%%%%%%%%%%%%%%%%%%%%%%%%%%%%%%%%%%%%%%%%%%%%%%%%%
\section{Short-time existence of the Yamabe flow} \label{shorttime}
%%%%%%%%%%%%%%%%%%%%%%%%%%%%%%%%%%%%%%%%%%%%%%%%%%%

Consider a manifold with bounded geometry $(M,g_{0})$ of dimension $m\geq 3$ and set $\eta := (m-2)/4$.
We write $\Delta_{g_{0}}$ for the negative Laplace Beltrami operator of $(M,g_{0})$.
In this section we construct a short-time solution to the Yamabe flow equation \eqref{yf-u} of the conformal factor
\begin{equation}\label{transformed2}
    \partial_{t}u = (m-1)u^{-1/\eta}\Delta_{g_{0}} u - \eta\scal(g_{0})u^{1-1/\eta},\;\; u|_{t=0} = 1.
\end{equation}
We plan to construct a solution as a fixed point of a contraction in $C^{k+2,\alpha}(M\times [0,T])$
and some short time $T>0$. We assume below that $k=0$, since the general case 
follows the $k=0$ case verbatim. We write $u=1+v$ and obtain from \eqref{transformed2} an equation for $v$:
\begin{equation}\label{linearizedflow}
    \partial_{t}v = (m-1)\Delta_{g_{0}} v (1 + v)^{-1/\eta} - \eta\scal(g_{0})(1+v)^{1-1/\eta}; \;\; v|_{t=0} = 0.
\end{equation}
Assume $v \in C^{2,\alpha}(M\times [0,T])$ with $\|v\|_{2,\alpha} \leq \mu$
for some $\mu <1$. Then the following series converges in the Banach space $C^{2,\alpha}(M\times [0,T])$:
\begin{align*}
        &(1 + v)^{-1/\eta}  =  \displaystyle\sum_{j=0}^{\infty} a_{j} v^{j} =  1 - \dfrac{v}{\eta} + \displaystyle\sum_{j=2}^{\infty} a_{j} v^{j} 
        =: 1 - \dfrac{v}{\eta} + v^{2}s(v) \\ &\textup{with} \quad \|(1+v)^{-1/\eta}\|_{2,\alpha} \leq C_{\mu} \quad
        \textup{and} \quad  \|s(v)\|_{2,\alpha} \leq C_{\mu},  
\end{align*}
for some $C_{\mu} > 0$, depending only on $\mu$.  
Plugging the identity $(1+v)^{-1/\eta} = 1 - v/\eta + v^{2}s(v)$ 
into (\ref{linearizedflow}) yields, after rescaling the time variable by $(m-1)$, 
the following flow equation:
\begin{align}\label{shorttimeflow}
\begin{split}
    (\partial_{t} - \Delta_{g_{0}}) v = &-\dfrac{1}{\eta}v\Delta_{g_{0}} v + v^{2}s(v)\Delta_{g_{0}} v 
       - \dfrac{\eta}{m-1}\scal(g_{0}) + \dfrac{1}{m-1}\scal(g_{0})v \\
       & + \dfrac{1}{m-1}\scal(g_{0})v^{2}(1 - \eta s(v)  -\eta vs(v)).
       \end{split}
\end{align}
We will simplify the right hand side by introducing two non-linear operators, 
the first one containing no derivatives of $v$:
\begin{equation*}\begin{split}
    F_{1}(v) := & - \dfrac{\eta}{m-1}\scal(g_{0})+ \dfrac{1}{m-1}\scal(g_{0})v
     \\ & + \dfrac{1}{m-1}\scal(g_{0})v^{2}(1 - \eta s(v)  -\eta vs(v)).
        \end{split} \end{equation*}
The second one is in a certain sense quadratic in $v$ and defined by
\begin{equation*}\begin{split}
F_{2}(v) := &-\dfrac{1}{\eta}v\Delta_{g_{0}} v + v^{2}s(v)\Delta_{g_{0}} v.
\end{split} \end{equation*}
In this notation, $(\ref{shorttimeflow})$ can be written as 
\begin{equation} \label{flowcontraction}
    (\partial_{t} - \Delta_{g_{0}})v = (F_{1}+F_{2})v; \;\; v|_{t=0} = 0.
\end{equation}
Our intention is to prove short-time existence of solution of (\ref{flowcontraction}) by using 
exactly the same argument as in \cite[Corollary 1.2]{giuseppebruno}, 
albeit with slightly different H\"older spaces, which gives some conditions for the contraction 
argument to work (the argument in \cite[Corollary 1.2]{giuseppebruno} does not 
depend on a specific choice of H\"older spaces). To this end, we need to prove some properties of $F_1$ and $F_2$.

\begin{lem} \label{lemmaf2}
Denote by $\mathscr{B}$ the open ball of radius $1$ in $C^{2,\alpha}(M\times [0,T])$.
Then the map $F_{2}:\mathscr{B} \rightarrow C^{\alpha}(M\times [0,T])$ is bounded.  Moreover, for any two functions $v$,$v' \in \mathscr{B} \subset C^{2,\alpha}(M\times [0,T])$ satisfying $$\|v\|_{2,\alpha},\|v'\|_{2,\alpha} \leq\mu < 1,$$ 
	there exists a constant $C_{\mu} > 0$ such that
	\begin{enumerate}
		\item $\|F_{2}(v) - F_{2}(v')\|_{\alpha} \leq C_{\mu}\max\{\|v\|_{2,\alpha},\|v'\|_{2,\alpha}\}\|v-v'\|_{2,\alpha}$,  
		\item $\|F_{2}(v)\|_{k,\alpha} \leq C_{\mu}\|v\|^{2}_{k+2,\alpha}.$
	\end{enumerate}
\end{lem}

\begin{proof}
We shall write $\Delta= \Delta_{g_{0}}$ for simplicity of notation.
 First, let $v \in \mathscr{B}$ with $\|v\|_{2,\alpha} \leq \mu < 1$.  
 Then, by the definition of $C^{2,\alpha}(M\times [0,T])$ and the fact that 
 $\Delta \in \mbox{Diff}^{2}_{\Phi}$(M), it follows that $\Delta v \in C^{\alpha}(M\times [0,T])$. We can thus estimate
$$\begin{array}{rcl}
\|F_{2}(v)\|_{\alpha} & \leq & C_\mu \left( \|v\Delta v\|_{\alpha} + \|v^{2}s(v)\Delta v\|_{\alpha}\right) \\
& \leq &C_\mu \left( \|v\|_{\alpha}\|\Delta v\|_{\alpha} + \|v^{2}s(v)\|_{\alpha}\|\Delta v\|_{\alpha}\right)\\
& \leq & C_\mu \|v\|^{2}_{2,\alpha},
\end{array}$$
for some $C_\mu>0$ depending only on $\mu$ and possibly changing in each estimation step. 
This proves the second item and in particular boundedness of $F_{2}:\mathscr{B} \rightarrow 
C^{\alpha}(M\times [0,T])$. For the first item we write for any 
$v,v' \in \mathscr{B}$
\begin{equation} \label{dif1}
    v^{2}s(v) - (v')^{2}s(v') =: (v-v')O_{1}(v,v'), 
\end{equation}
where $O_{1}(v,v')$ is a polynomial combination in $v$ and $v'$. \Eqref{dif1} implies
\begin{align*}
    F_{2}(v) - F_{2}(v')
= & -\frac{1}{\eta}\left(\Delta v (v-v') + v'\Delta(v-v')\right) \\
& + O_{1}(v,v')\left((v - v')\Delta v + v'\Delta (v-v')\right),
\end{align*}
which then implies
\begin{align*}
\|F_{2}(v) - F_{2}(v')\|_{\alpha} & \leq C_\mu \left( 
\|\Delta v\|_{\alpha}\|v-v'\|_{\alpha} + \|v'\|_{\alpha}\|\Delta(v-v')\|_{\alpha} \right.\\
& + \left.\|\Delta v\|_{\alpha}\|v-v'\|_{\alpha} + \|v'\|_{\alpha}\|v-v'\|_{2,\alpha} \right) \\
&\leq C_{\mu}\max\{\|v\|_{2,\alpha},\|v'\|_{2,\alpha}\}\|v-v'\|_{2,\alpha}.
\end{align*}\end{proof} 

\begin{lem} \label{lemmaf1}
Assume that $\scal(g_{0}) \in C^{1,\alpha}(M)$.  
Denote by $\mathscr{B}$ the open ball of radius $1$ in $C^{2,\alpha}(M\times [0,T])$.
Then $F_{1}$ maps $\mathscr{B}$ into $C^{1,\alpha}(M\times [0,T])$.  Furthermore, if 
$v$,$v' \in \mathscr{B}$ with $\|v\|_{2,\alpha}$,
$\|v'\|_{2,\alpha} \leq \mu < 1$, there exists a constant $C_{\mu}$ such that 
\begin{enumerate}
\item $\|F_{1}(v) - F_{1}(v')\|_{1,\alpha} \leq C_{\mu}\|v-v'\|_{2,\alpha}$, 
\item $\|F_{1}(v)\|_{1,\alpha} \leq C_{\mu}.$
\end{enumerate}
\end{lem}

\begin{proof}
 First, consider $v \in \mathscr{B} \subset C^{2,\alpha}(M\times [0,T])$.  
 Since by assumption $\scal(g_{\Phi}) \in C^{1,\alpha}(M)$, we find 
$$
\scal(g_{0})v^{2}(1 - \eta s(v)  -\eta vs(v)) \in C^{1,\alpha}(M\times [0,T]).
$$  Now, assume $\|v\|_{2,\alpha} \leq \mu < 1$.  We can now estimate
\begin{align*}
\|F_{1}(v)\|_{1,\alpha} &\leq C_{\mu} \|\scal(g_{0})\|_{1,\alpha}
\left( 1 + \|v\|_{2,\alpha} + \|v^2\|_{2,\alpha}\right) \leq C_{\mu},
\end{align*}
for some $C_\mu>0$ depending only on $\mu$ and possibly changing in each estimation step. 
This completes the proof for the second item. In particular, 
$F_{1}$ indeed maps $\mathscr{B}$ into $C^{1,\alpha}(M\times [0,T])$.
For the first item we have for any $v,v' \in \mathscr{B}$ with
$\|v\|_{2,\alpha},\|v'\|_{2,\alpha} \leq \mu < 1$
\begin{align*}
    \|F_{1}(v) - F_{1}(v')\|_{1,\alpha} &\leq C_\mu \|\scal(g_{0})\|_{1,\alpha}\|v-v'\|_{2,\alpha} \\
    &+C_\mu \|\scal(g_{0})\|_{1,\alpha}\|v^{2} - (v')^{2}\|_{2,\alpha}\\
&+ C_\mu \|\scal(g_{0})\|_{1,\alpha}\|v^{2}s(v) - (v')^{2}s(v')\|_{2,\alpha} \\
&+C_\mu \|\scal(g_{0})\|_{1,\alpha}\|v^{3}s(v)-(v')^{3}s(v')\|_{2,\alpha} \\
&\leq C_{\mu}\|v-v'\|_{2,\alpha},
\end{align*}
where in the final estimate we use \eqref{dif1} and its analogue for $v^{3}s(v)$.
This concludes the first item and, naturally, finishes the proof.
\end{proof} 

Now, exactly the same argument as in \cite[{Corollary 1.2}]{giuseppebruno} (with $a=1$) implies directly that, 
for $\scal(g_{0}) \in C^{1,\alpha}(M)$,
the map $H \circ (F_1 + F_2)$ is a contraction on a closed 
ball $\overline{\mathscr{B}}_\mu \subset C^{2,\alpha}(M\times [0,T])$
of radius $\mu>0$, provided $\mu, T > 0$ are sufficiently small. Thus the flow
\eqref{flowcontraction} admits a solution $v \in \overline{\mathscr{B}}_\mu$
as a fixed point of that contraction. Setting $u=1+v$, we obtain a short-time
solution for the Yamabe flow \eqref{transformed2} and thus to \eqref{yamabeflow}.
The same argument yields a solution in $C^{k+2,\alpha}(M\times [0,T])$
for a general $k \in \N_0$, provided $\scal(g_{0}) \in C^{k+1,\alpha}(M)$.

\begin{thm} \label{theoremshorttime}
Consider a manifold with bounded geometry $(M,g_{0})$ of dimension $m \geq 3$.
Assume $\scal(g_{0}) \in C^{k+1,\alpha}(M)$ for some $\alpha \in (0,1)$ and any $k \in \N_0$.  
Then the Yamabe flow \eqref{yamabeflow} admits a unique solution $g = u^{4/(m-2)}g_{0}$, where 
$u \in C^{k+2,\alpha}(M\times [0,T])$, for some $T>0$
sufficiently small.
\end{thm}

%%%%%%%%%%%%%%%%%%%%%%%%%%%%%%%%%%%%%%%%%%%%%%%%%%%%%
\section{Curvature-normalized Yamabe flow ($\mbox{CYF}^{+}$)} \label{renormalization}
%%%%%%%%%%%%%%%%%%%%%%%%%%%%%%%%%%%%%%%%%%%%%%%%%%%%%%%%

Consider the increasing curvature normalized Yamabe flow $\mbox{CYF}^{+}$
\begin{equation*}
    \partial_{t}g = (\scal(g)_{\sup} - \scal(g))g, \;\; \mbox{where} \;\; \scal(g(t))_{\sup} := \displaystyle\sup_{M} \scal(g(t)).
\end{equation*}
see \eqref{CYF}, introduced by Su\'{a}rez-Serrato and Tapie \cite{serratotapie} to study entropy rigidity on the Yamabe flow
in the compact setting. We are interested in the non-compact setting of a manifold with bounded geometry $(M,g_{0})$, which is why the usual normalization by \eqref{rhocompact} does not work
and we resort to the $\mbox{CYF}^{+}$ normalization.
We can study the decreasing curvature normalized Yamabe flow $\mbox{CYF}^{-}$
with $\scal(g)_{\sup}$ replaced by $\scal(g)_{\inf}$
along the same lines.\medskip

Short time existence of $\mbox{CYF}^{+}$ (as well as $\mbox{CYF}^{-}$) follows
by a simple time rescaling. Indeed, let $g(t) = u(t)^{1/\eta}g_{0}$ be family of Riemannian metrics satisfying the 
(unnormalized) Yamabe flow \eqref{yamabeflow} with $u \in C^{2,\alpha}(M\times [0,T])$. Consider the functions
\begin{equation}
\begin{split}
    &f(t) = \exp \left( \int_{0}^{t}\eta\scal(g(\theta))_{\sup}\di \theta\right), \\
    &F(t) = \int_{0}^{t}f(\theta)^{1/\eta}\di \theta - f(0)^{1/\eta}.
\end{split}
\end{equation}
Note that $f$ is positive and $F$ is a primitive for $f$ satisfying $F(0)=0$.  Moreover, since $dF/dt > 0$, it follows that $F^{-1}$ is well-defined.  Thus, we can define a 1-parameter family of Riemannian metrics by
\begin{equation}
    \tilde{g}(\tau) := \tilde{u}(\tau)^{1/\eta} g_{0},  \quad \textup{where} \quad \tilde{u}(\tau):= (fu)(F^{-1}(\tau)).
\end{equation}
One can easily check from $u \in C^{2,\alpha}(M\times [0,T])$ that $\tilde{u} \in C^{2,\alpha}(M\times [0,\widetilde{T}])$
with $\widetilde{T} = \max F$. Moreover, it follows from direct computations that 
\begin{equation}
    \partial_{\tau} \tilde{g} = \Bigl(\scal(\tilde{g})_{\sup} - \scal(\tilde{g})\Bigr)\tilde{g}, \;\; \tilde{g}(0) = g_{0}.
\end{equation}
It is also possible to invert the process and obtain a solution of the standard Yamabe flow, proving said relation.
This proves short time existence of the increasing curvature normalized Yamabe flow $\mbox{CYF}^{+}$ 
(and similarly $\mbox{CYF}^{-}$).

%%%%%%%%%%%%%%%%%%%%%%%%%%%%%%%%%%%%%%%%%%%%%%%%%%%%%%%%%%%%%%%%%%%
\section{Evolution of the scalar curvature along $\mbox{CYF}^{+}$} \label{evolutionsection}
%%%%%%%%%%%%%%%%%%%%%%%%%%%%%%%%%%%%%%%%%%%%%%%%%%%%%%%%%%%%%%%%%%%%%

\noindent We begin with an easy observation. 

\begin{lem} \label{conformlaplace}
    Let $M$ be any $m$-dimensional smooth manifold.  Given any two Riemannian metrics 
    $g$ and $\tilde{g}$ on $M$ related by a conformal transformation $\tilde{g} = u^{1/\eta}\cdot g$ for some 
    positive $u \in C^{2}(M)$, then for any $f \in C^{2}(M)$
    \begin{equation*}
        \Delta_{\tilde{g}}f = \dfrac{1}{u^{1/\eta}}\cdot \Delta_{g}f + \dfrac{2}{u^{1+1/\eta}}\cdot g(\nabla f, \nabla u).
    \end{equation*}
\end{lem}
\begin{proof}
 First, note that $\tilde{g}^{-1} = u^{-1/\eta}\cdot g^{-1}$ and $\sqrt{|\det \tilde{g}|} = u^{m/2\eta}\cdot\sqrt{|\det g|}$.  Thus,
 we compute in local coordinates
\begin{align*} 
    \Delta_{\tilde{g}}f = &\dfrac{1}{\sqrt{|\det \tilde{g}}|}\cdot \displaystyle\sum_{j}\partial_{j}\left(\sqrt{|\det \tilde{g}|}\cdot \displaystyle\sum_{i}\tilde{g}^{ij}\cdot \partial_{i}f\right) \\
    = &\dfrac{1}{u^{m/2\eta}\cdot \sqrt{|\det g|}}\displaystyle\sum_{j}\partial_{j}\left(u^{(m-2)/2\eta}\cdot \sqrt{|\det g|}\cdot \displaystyle\sum_{i}g^{ij}\cdot \partial_{i}f\right) \\
    = &\dfrac{1}{u^{1/\eta}\cdot \sqrt{|\det g|}}\cdot \displaystyle\sum_{j}\partial_{j}\left(\sqrt{|\det g|}\cdot \displaystyle\sum_{i}g^{ij}\cdot \partial_{i}f\right) \\
    &+ \dfrac{2}{u^{1+1/\eta}}\cdot \displaystyle\sum_{i,j}g^{ij}\cdot \partial_{j} u\cdot \partial_{i}f 
    = \dfrac{1}{u^{1/\eta}}\cdot \Delta_{g}f + \dfrac{2}{u^{1+1/\eta}}\cdot g(\nabla f, \nabla u).
    \end{align*}
\end{proof}

\noindent Note that the increasing curvature normalized Yamabe flow \eqref{CYF} can be rewritten as
(recall $\scal(g)_{\sup}$ denotes the supremum of $\scal(g)$)
\begin{equation}\label{simple-u-eq}
\frac{1}{\eta}\partial_t u = \Bigl(\scal(g)_{\sup} - \scal(g)\Bigr) u.
\end{equation}
From here we conclude immediately
\begin{equation} \label{eq2}
	\begin{split}
	\frac{1}{\eta} \partial_{t}(u^{-1}\Delta_{g_{0}}u) &= - \frac{1}{\eta} u^{-2} \, \partial_{t}u \cdot \Delta_{g_{0}}u + 
	\frac{1}{\eta} u^{-1} \Delta_{g_{0}} (\partial_{t}u) \\
	&= -u^{-1}\bigl(\scal(g)_{\sup}-\scal(g)\bigr) \cdot \Delta_{\Phi}u \\
	&+ u^{-1} \Delta_\Phi \Bigl( \bigl(\scal(g)_{\sup}-\scal(g)\bigr) u \Bigr)
	\\ &=  u^{-1} \Bigl( \scal(g) \cdot \Delta_{g_{0}}u - \Delta_\Phi \bigl( \scal(g) u\bigr) \Bigr).
	\end{split}
\end{equation}
Moreover, from Lemma \ref{conformlaplace} we obtain
\begin{align*}
    u^{-1}\Delta_{g_{0}}(\scal(g) u) = & u^{-1}\scal(g)\Delta_{g_{0}}u+ \Delta_{\Phi}\scal(g) +2u^{-1}g_{0}(\nabla u,\nabla \scal(g)) \\
    = & u^{-1}\scal(g)\Delta_{g_{0}}u + u^{1/\eta}\Delta_{g}\scal(g),
\end{align*}
where $\Delta_g$ is the negative Laplace Beltrami operator of the conformally
transformed metric $g=u^{1/\eta}\cdot g_{0}$.
Combined with (\ref{eq2}) this gives
\begin{equation} \label{eq3}
    \dfrac{1}{\eta}\partial_{t}(u^{-1}\Delta_{g_{0}}u) = -u^{1/\eta}\Delta_{g}\scal(g).
\end{equation}
On the other hand, from \eqref{CYF} it is also straightforward that 
\begin{equation} \label{eq4}
    \partial_{t}u^{-1/\eta} = u^{-1/\eta}(\scal(g)-\scal(g)_{\sup}).
\end{equation}  
Finally, combining \eqref{eq2} and \eqref{eq4} with the transformation formula for the scalar curvature, cf. \eqref{scal-trafo}
\begin{equation} \label{conftrans}
    \scal(g(t)) = \scal(u^{1/\eta} g_{0}) =  -u^{-1/\eta}\left[\dfrac{m-1}{\eta}u^{-1}\Delta_{g_{0}} u - \scal(g_{0}) \right],
\end{equation}
it provides us the expression
\begin{equation} \label{derivscalar}
    \partial_{t}\scal(g) = (m-1)\Delta_{g}\scal(g) + \scal(g)(\scal(g)-\scal(g)_{\sup}).
\end{equation}
Based on \eqref{derivscalar}, we can now prove the following

\begin{lem} \label{evolutionscalar}
    Suppose $\scal(g_{0}) \in C^{4,\alpha}(M)$ is negative and bounded away from 
    zero\footnote{In fact, boundedness away from zero for the scalar curvature will only become important
    in the next section, but we list it here as a condition for consistency.}, 
    that is, there are constants $a_{1},a_{2} >0$ such that 
    \begin{equation}
        -\infty < -a_{1} \leq \scal(g_{0}) \leq -a_{2} < 0.
    \end{equation}
    Then along $\mbox{CYF}^{+}$ with positive solution $u \in C^{4,\alpha}(M\times [0,T])$,
supremum of the the scalar $\scal(g(t))_{\sup}=\displaystyle\sup_{M} \scal(g(t))$ is non-increasing. 
\end{lem}

\begin{proof}
By Theorem \ref{theoremshorttime}, $\mbox{CYF}^{+}$ exists for short time in $C^{4,\alpha}(M \times [0,T])$.
From the transformation rule of the scalar curvature \eqref{conftrans}, it follows that $\scal(g) \in C^{2,\alpha}(M\times [0,T])$. 
Applying the arguments of \S \ref{omori} to $\scal(g)$, we conclude from \eqref{derivscalar} by Proposition \ref{envelope}, similar to 
Corollary \ref{diff-ineq}, for almost all $t \in [0,T]$
\begin{equation}\label{maxevolut}
        \partial_{t} \scal(g(t))_{\sup} \leq \scal(g(t))_{\sup}(\scal(g(t))_{\sup} - \scal(g(t))_{\sup}) = 0.
    \end{equation}
This implies directly that $\scal(g)_{\sup}$ is non-increasing along $\mbox{CYF}^{+}$.
    \end{proof}

Knowing that the supremum of the scalar curvature is non-increasing in time, 
the next results shows that the scalar curvature approaches its supremum at an exponential rate.

\begin{lem} \label{exponentialapproxim}
Suppose $\scal(g_{0}) \in C^{4,\alpha}(M)$ is negative, bounded away from zero
as in Lemma \ref{evolutionscalar}. Then along $\mbox{CYF}^{+}$
with a positive solution $u \in C^{4,\alpha}(M\times [0,T])$
\begin{equation*}
        \|\scal(g(t))_{\inf} - \scal(g(t))_{\sup}\|_{\infty} \leq Ce^{\scal(g_{0})_{\sup}t},
    \end{equation*}
with $C>0$ a constant independent of $T$, where $\scal(g(t))_{\inf} := \inf_M \scal(g(t))$.
\end{lem}

\begin{proof}
Applying the arguments of \S \ref{omori} to $\scal(g)$, we conclude from \eqref{derivscalar} by Proposition \ref{envelope}, similar to 
Corollary \ref{diff-ineq}, for almost all $t \in [0,T]$
\begin{equation} \label{minevolut}
        \partial_{t}\scal(g)_{\inf} \geq \scal(g)_{\inf}(\scal(g)_{\inf} - \scal(g)_{\sup}). 
 \end{equation}
From here it follows that $\scal(g)_{\inf}$ is non-decreasing along the $\mbox{CYF}^{+}$. 
Combining \eqref{minevolut} with \eqref{maxevolut}, we find
    \begin{align*}
        \partial_{t}(\scal(g)_{\sup}-\scal(g)_{\inf}) &\leq - \scal(g)_{\inf}(\scal(g)_{\inf}-\scal(g)_{\sup}) \\
        &= \scal(g)_{\inf}(\scal(g)_{\sup}-\scal(g)_{\inf})\\
        &\leq \scal(g_{0})_{\sup}(\scal(g)_{\sup}-\scal(g)_{\inf}).
    \end{align*}
    Integrating both sides of the last inequality gives 
    \begin{equation}
        (\scal(g)_{\sup}-\scal(g)_{\inf})(t) \leq Ce^{\scal(g_{0})_{\sup}t},
    \end{equation}
    where $C$ depends only on the initial data.  This means that the difference between the supremum and the infimum 
    of the scalar curvature decreases exponentially along the flow.  Consequently, the scalar curvature approaches 
    $\scal(g)_{\sup}$ at an exponential rate too, therefore implying the desired outcome.
\end{proof}

%%%%%%%%%%%%%%%%%%%%%%%%%%%%%%%%%%%%%%%%%%%%%%%%%%%%%%%%%%%
\section{Uniform estimates along $\mbox{CYF}^{+}$} \label{uniformsection}
%%%%%%%%%%%%%%%%%%%%%%%%%%%%%%%%%%%%%%%%%%%%%%%%%%%%%%%%%%%%%%

\noindent We start immediately with the central result of the section. If we assume $\scal(g_{0}) \in C^{4,\alpha}(M)$, then 
the solution $u \in C^{4,\alpha}(M\times [0,T'])$ of $\mbox{CYF}^{+}$ exists by Theorem \ref{theoremshorttime} for $T'>0$ sufficiently small.
Assume $u$ in fact exists in $C^{4,\alpha}(M\times [0,T))$ on a larger time interval $[0,T)$ with maximal time $T\geq T'$. Then even in the
maximal time interval $[0,T)$ we obtain $T$-independent a priori estimates.

\begin{thm}\label{uniformboundu}
    Assume $\scal(g_{\Phi}) \in C^{4,\alpha}(M)$ is negative and bounded away from zero as in Lemma \ref{exponentialapproxim}. 
    Let $u \in C^{4,\alpha}(M\times [0,T))$ to be the solution of $\mbox{CYF}^{+}$ extended to a maximal time interval $[0,T)$.
    Then there exist constants $c_{1},c_{2}>0$, depending on $u(0),\sup |\scal(g_{0})|$ and $\inf|\scal(g_{0})|$, 
    and independent of $T$, such that 
    \begin{equation*}
        0<c_{1}\leq u(p,t) \leq c_{2}, \;\; \mbox{for all} \;\; (p,t) \in M\times [0,T).
    \end{equation*}
\end{thm}

\begin{proof}
First, we consider the flow for a short time interval $[0,T']$, where $u$ is guaranteed to be positive. 
The estimates below will show that $u$ stay positive, bounded away from zero uniformly on $[0,T']$
and thus all of the arguments hold on the maximal interval $[0,T)$. By the differential inequalities
in Proposition \ref{diff-ineq} we have (a priori almost everywhere on $[0,T']$, however as just explained a posteriori 
almost everywhere on the full time interval)
  \begin{equation}
  \begin{split}
       &\frac{\partial}{\partial t} u_{\inf}
       \geq \eta \sup_M \scal(g(t)) \cdot u_{\inf} + \eta \inf_M |\scal(g_{0})| \cdot u_{\inf}^{1-1/\eta}, \\
       &\frac{\partial}{\partial t} u_{\sup}
       \leq \eta \sup_M \scal(g(t)) \cdot u_{\sup} + \eta \sup_M |\scal(g_{0})| \cdot u_{\sup}^{1-1/\eta}.
\end{split}
   \end{equation}
Multiplying both sides of the first inequality by $\frac{1}{\eta} u^{1/\eta-1}_{\inf}$, 
and of the second inequality by $\frac{1}{\eta} u^{1/\eta-1}_{\sup}$, we obtain
  \begin{equation}\label{w12}
  \begin{split}
          &\frac{\partial}{\partial t} u^{1/\eta}_{\inf}
       \geq \sup_M \scal(g(t)) \cdot u^{1/\eta}_{\inf} + \inf_M |\scal(g_{0})|, \\
       &\frac{\partial}{\partial t} u^{1/\eta}_{\sup}
       \leq \sup_M \scal(g(t)) \cdot u^{1/\eta}_{\sup} + \sup_M |\scal(g_{0})|.
            \end{split}
   \end{equation}
Write $\omega_{1} := u^{1/\eta}_{\inf}$ and $\omega_{2} := u^{1/\eta}_{\sup}$. 
We obtain from \eqref{w12}
  \begin{equation}\label{w123}
  \begin{split}
          &\frac{\partial}{\partial t} \omega_1
       \geq \inf_M \scal(g_{0}) \cdot \omega_1 + \inf_M |\scal(g_{0})| =: b\omega_{1} + a, \\
       &\frac{\partial}{\partial t} \omega_2
       \leq \sup_M \scal(g_{0}) \cdot \omega_2 + \sup_M |\scal(g_{0})| =: B\omega_{2} + A,
            \end{split}
   \end{equation}
where in the first inequality we used the fact that by \eqref{minevolut} $\scal(g)_{\inf}$ is non-decreasing in time, 
while the second inequality from Lemma \ref{evolutionscalar}, since $\scal(g)_{\sup}$ is non-increasing in time.  
\medskip

The first inequality is equivalent to $(e^{-bt}\omega_{1})'\geq ae^{-bt}$.  Hence, integration on both sides over $[0,t]$ gives 
the following estimate
    \begin{align*}
        &\omega_{1}(t) \geq e^{bt}\omega_{1}(0) + \dfrac{a}{b}(e^{bt}-1) \\
        \Longleftrightarrow \quad &u^{1/\eta}_{\inf}(t) \geq u^{1/\eta}_{\inf}(0)e^{ \inf_M \scal(g_{0}) \cdot t}
         + \dfrac{ \inf_M |\scal(g_{0})|}{\inf_M \scal(g_{0})}\;(e^{ \inf_M \scal(g_{0}) \cdot t}-1)
        \\  \Longleftrightarrow \quad &u^{1/\eta}_{\inf}(t) \geq u^{1/\eta}_{\inf}(0)e^{ \inf_M \scal(g_{0}) \cdot t}
         + \dfrac{ \inf_M |\scal(g_{0})|}{\sup_M |\scal(g_{0})|}\;(1-e^{\inf_M \scal(g_{0}) \cdot t}) %\\
         %&\qquad \quad \geq \dfrac{ \inf_M |\scal(g_\Phi)|}{\sup_M |\scal(g_\Phi)|}\;(1-e^{\inf_M \scal(g_\Phi) \cdot t})
    \end{align*}
Hence, by setting the right-hand side as a function $f(t)$, it follows that $u^{1/\eta}_{\inf}(t) \geq f(t)$, with 
\begin{align*}
    &f(t) := c' + (c-c') e^{-|d|\cdot t},\;\; \mbox{with}\\
    &c = u^{1/\eta}_{\inf}(0),\;\; c'=\dfrac{ \inf_M |\scal(g_{0})|}{\sup_M |\scal(g_{0})|} \;\; \mbox{and}\;\; d = \scal(g_{0})_{\inf}.
\end{align*}
%for $c = u^{1/\eta}_{\inf}(0)>0$, $c'= \inf_{M}|\scal(g_{\Phi})|/\sup_{M}|\scal(g_{\Phi})|>0$ and $d = \scal(g_{\Phi})_{\inf}$. 
If $c\geq c'$ then $f(t) \geq c'$, since $e^{-|d|\cdot t}\geq 0$.  On the other hand, if $c \leq c'$ then it follows that $f(t) = c' + (c'-c)(-e^{-|d|\cdot t}) \geq c$, since $(-e^{-|d|\cdot t}) \geq -1$.  Hence, 
\begin{equation*}
    u^{1/\eta}_{\inf}(t) \geq f(t) \geq \min\left\{u^{1/\eta}_{\inf}(0),\dfrac{ \inf_M |\scal(g_{0})|}{\sup_M |\scal(g_{0})|}\right\}>0.
\end{equation*}
This yields \textit{a priori} positive lower bound for $u$.  On the other hand, if $c = c'$, then it follows straightforwardly that $u^{1/\eta}_{\inf}(0)$ is \textit{a priori} lower bound for $u^{1/\eta}_{\inf}(t)$.
Now, let us turn our attention to the second equation in \eqref{w123}. 
This inequality is equivalent $(e^{Bt}\omega_{2})' \leq Ae^{-Bt}$, which after integrating over $[0,t]$ implies 
    \begin{align*}
       &\omega_{2}(t) \leq \omega_{2}(0)e^{Bt} - \dfrac{A}{B}(1-e^{Bt}) \leq \omega_{2}(0)   \dfrac{A}{B}(1-e^{Bt}) \\
        \Longleftrightarrow \quad &u^{1/ \eta}_{\sup}(t) \leq u^{1/\eta}_{\sup}(0) + \dfrac{\sup_M |\scal(g_{0})|}{\inf_M |\scal(g_{0})|}\;(1-e^{\sup_M \scal(g_{0}) \cdot t})
        %\\ & \qquad \quad \leq u^{1/\eta}_{\sup}(0),
    \end{align*}
Proceeding along the lines of the estimate for $u^{1/\eta}_{\inf}(t)$, consider the right-hand side as a function $F(t)$, where 
\begin{align*}
    &F(t) := C + C'(1 - e^{-|D|\cdot t}),\;\; \mbox{with}\\ &C = u^{1/\eta}_{\sup}(0),\;\; C' = \dfrac{\sup_M |\scal(g_{0})|}{\inf_M |\scal(g_{0})|} \;\;\mbox{and}\;\; D = \scal(g_{0})_{\sup}.
\end{align*}
Note that $- e^{-|D|\cdot t} \leq 0$.  Hence $F(t) \leq C + C'$ and, therefore, 
\begin{equation*}
    u^{1/\eta}_{\sup}(t) \leq F(t) \leq  u^{1/\eta}_{\sup}(0) + \dfrac{\sup_M |\scal(g_{0})|}{\inf_M |\scal(g_{0})|} < +\infty,
\end{equation*}
concluding the proof.
\end{proof}

\begin{prop} \label{uniformboundpartialu}
   Assume $\scal(g_{0}) \in C^{4,\alpha}(M)$ is negative and bounded away from zero as in Lemma \ref{exponentialapproxim}. 
   Let $u \in C^{4,\alpha}(M\times [0,T))$ to be the solution of $\mbox{CYF}^{+}$ extended to a maximal time interval $[0,T)$.  
   Then there exists a constant $C>0$, depending on $u(0),\sup |\scal(g_{0})|$ and $\inf|\scal(g_{0})|$, 
    and independent of $T$, such that
    \begin{equation}
        \|\partial_{t}u\|_{\infty} \leq Ce^{\sup_M \scal(g_{0}) \cdot t}.
    \end{equation}
\end{prop}

\begin{proof}
The $\mbox{CYF}^{+}$ flow \eqref{CYF} can be rewritten as (cf. \eqref{simple-u-eq})
\begin{equation}
\frac{1}{\eta}\partial_t u = \bigl(\scal(g)_{\sup} - \scal(g)\bigr) u.
\end{equation}
Then, employing Lemma \ref{exponentialapproxim} and Theorem \ref{uniformboundu}, it follows directly that 
    \begin{align*}
        \|\partial_{t}u\|_{\infty} &\leq |\eta|\|\scal(g)_{\sup}-\scal(g)\|_{\infty} \|u\|_{\infty} \\
        &\leq C e^{\sup_M \scal(g_{0}) \cdot t}.
    \end{align*}
    \end{proof}
    
Now we are ready to convert the a priori estimates in Theorem \ref{uniformboundu}
into uniform H\"older regularity on $[0,T]$, where $[0,T)$ is the maximal time interval, 
where the $\mbox{CYF}^{+}$ flow solution $u$ exists in $C^{4,\alpha}(M\times [0,T))$.  In fact, recall that the solution $u$ of the $\mbox{CYF}^{+}$ must satisfy the equation in \eqref{CYF-u}, that is,
\begin{align*}
&\partial_{t}u(t) - (m-1)u(t)^{-1/\eta}\Delta_{g_{0}} u(t) \\ 
&= \eta \Bigl(\sup_M \scal(g(t)) \cdot u(t) - \scal(g_{0})u(t)^{1-1/\eta}\Bigr) =: f.
\end{align*}
Proceeding similarly to the proof of Proposition \ref{Q-mapping-properties}, we can employ once again the argument to local $\delta$-balls and write
\begin{align}\label{local-heat}
\Bigl(\partial_{t} - a \cdot \widetilde{\Delta_{g_{0}}}\Bigr) \Psi^* u =  \Psi^* f,
\end{align}
where $a:= (m-1)\Psi^{*}u^{-1/\eta}$.  Under the assumption that $\scal(g_{0}) \in C^{4,\alpha}(M)$ and is negative and bounded away from zero, it follows from Theorem \ref{uniformboundu} and Proposition \ref{uniformboundpartialu} that $a$ is bounded in $Q_{\delta}$.  Thus, the same argument as the one presented in Proposition \ref{Q-mapping-properties} implies that $u \in C^{\alpha}(M\times [0,T])$.  Moreover, the estimate for $\|u\|_{\alpha}$ does not depend on the maximal time $T$.  Hence, we conclude

\begin{prop}\label{krylov-safonov-lemma}
Assume $\scal(g_{0}) \in C^{4,\alpha}(M)$ is negative and bounded away from zero. 
Let $u \in C^{4,\alpha}(M\times [0,T))$ be the solution of $\mbox{CYF}^{+}$ extended to a maximal time interval $[0,T)$. 
Then $u \in C^{\alpha}(M\times [0,T])$ with $T$-independent H\"older norm.
\end{prop}

This first gain in H\"older regularity can now be converted into higher order regularity
by standard parabolic Schauder estimates as in Proposition \ref{Q-mapping-properties}.
\begin{prop}\label{uniform-regularity}
Assume $\scal(g_{0}) \in C^{4,\alpha}(M)$ is negative and bounded away from zero. 
Let $u \in C^{4,\alpha}(M\times [0,T))$ be the solution of $\mbox{CYF}^{+}$ extended to a maximal time interval $[0,T)$. 
Then $u \in C^{4,\alpha}(M\times [0,T])$ with $T$-independent H\"older norm.
\end{prop}

\begin{rmk}\label{24}
Note that the argument from Proposition \ref{Q-mapping-properties} shows that in fact, if $\scal(g_{0}) \in C^{k,\alpha}(M)$ with $k \geq 4$ is negative and bounded away from zero,
the $\mbox{CYF}^{+}$ flow solution $u \in C^{2,\alpha}(M\times [0,T])$ on any time interval $[0,T]$ is in fact in $C^{k,\alpha}(M\times [0,T])$.
\end{rmk}

%%%%%%%%%%%%%%%%%%%%%%%%%%%%%%%%%%%%%%%%%%%%%%%%%%%%%
\section{Global existence of the $\mbox{CYF}^{+}$} \label{longtimesection}
%%%%%%%%%%%%%%%%%%%%%%%%%%%%%%%%%%%%%%%%%%%%%%%%%%%%%%%

We prove global existence of the flow, i.e. $u \in C^{4,\alpha}(M\times [0,\infty))$ by a contradiction. 
Assume the maximal time $T>0$ is finite. In that case we will now restart the flow at $t=T$, which contradicts
maximality of $T$. Restarting the flow at $t=T$ means constructing a solution $u'$ to the (unnormalized) Yamabe flow 
equation \eqref{yf-u} with initial condition $u'(0) = u(T)$. A rescaling of the time function, as in \S \ref{renormalization},
yields short time existence of the curvature normalized Yamabe flow. \medskip

Let us simplify notation by writing $u_0 =u(T)$ and $\Delta=\Delta_{g_{0}}$.
We linearize \eqref{yf-u} by setting $u'= u_0 + v$ for its solution with initial condition $u'(0) = u_0$.
We obtain from the second equation in \eqref{yf-u} 
 \begin{equation}\label{linearized-at-T}
        \left(\partial_{t}-(m-1)u_0^{-1/\eta}\Delta\right)v = F_{1}(v) + F_{2}(v); \;\; v|_{t=0} = 0,
    \end{equation}
    where we have abbreviated 
    \begin{align*}
            F_{1}(v) = Q_2(v), \quad
            F_{2}(v) &= (m-1)u_{0}^{-1/\eta}\Delta u_{0} 
            - \scal(g_{0})u_{0}^{1-1/\eta} + Q_1(v),
    \end{align*}
The terms $Q_1(v)$ include linear combinations of $v$ with coefficients given in terms of  
$u_{0}$ and $\Delta u_{0}$. The terms $Q_2(v)$ include quadratic combinations
of $v$ and $\Delta v$ with coefficients given again in terms of $u_{0}$ and $\Delta u_{0}$.
\medskip

Note that by Proposition \ref{uniform-regularity}, $u_0 \in C^{4,\alpha}(M)$.
Thus, $F_1$ contains quadratic combinations of $v$ and $\Delta v$, and $F_2$ 
$-$ linear combinations of $v$; with coefficients being in both cases elements of $C^{2,\alpha}(M \times [0,T'])$.
\medskip

Before we can establish short time existence of $v$ by setting up a fixed point as in \S  \ref{shorttime}, 
we note a general result from parabolic Schauder theory. This is basically a non-constructive analogue of
\cite[\S 4]{bahuaud2019long}.

\noindent We can now conclude with the proof of Theorem \ref{longtime}.

\begin{cor}
Assume $\scal(g_{0}) \in C^{k,\alpha}(M)$ is negative and bounded away from zero with $k\geq 4$. 
Then the increasing curvature normalized Yamabe flow $\mbox{CYF}^{+}$ exists for all times with 
conformal factor $u \in C^{k,\alpha}(M\times [0,\infty))$.
\end{cor}

\begin{proof}
Using Proposition \ref{Q-mapping-properties}, we can construct a solution 
$v\in C^{2,\alpha}(M \times [0,T'])$ to \eqref{linearized-at-T} for some $T'>0$ sufficiently
small, as a fixed point of 
\begin{equation}
Q \circ (F_1+F_2): C^{2,\alpha}(M\times [0,T']) \rightarrow C^{2,\alpha}(M\times [0,T']),
\end{equation}
in the same way as in \S \ref{shorttime}. Rescaling time as in  \S \ref{renormalization},
we obtain a solution $u \in C^{2,\alpha}(M\times [0,T + \varepsilon])$ to $\mbox{CYF}^{+}$,
with $\varepsilon > 0$ sufficiently small. Finally, the arguments of Proposition \ref{uniform-regularity},
cf. Remark \ref{24} imply that $u \in C^{k,\alpha}(M\times [0,T + \varepsilon])$ with $T$-independent H\"older norm. This contradicts
maximality of $T>0$ and hence the flow exists for all times. 
\end{proof}

%%%%%%%%%%%%%%%%%%%%%%%%%%%%%%%%%%%%%%%%%%%%%
\section{Convergence of the $\mbox{CYF}^{+}$}\label{section-convergence}
%%%%%%%%%%%%%%%%%%%%%%%%%%%%%%%%%%%%%%%%%%%%%%%

In this last section, we will assume the manifold with bounded geometry $(M,g)$ to be the open interior of a compact manifold with boundary $\overline{M}$.  This means there is a global defining function $x \in C^{\infty}(\overline{M})$ such that
$$\partial \overline{M} = \setdef{p \in M}{ x(p) = 0}$$
and $\di x \neq 0$ on $\partial \overline{M}$.  In order to present the convergence of the $\mbox{CYF}^{+}$, we must use a compact embedding of (weighted) H\"older spaces, where the weight is defined in terms of the function $x$.

\begin{defin}
The weighted H\"older space $x^\gamma C^{k,\alpha}(M)$ is defined as the space of functions 
$u= x^\gamma v$ with $v \in C^{k,\alpha}(M)$ and the norm $\|u\|_{k,\alpha,\gamma} := \|v\|_{k,\alpha}$.
\end{defin}

\noindent We now obtain the following compactness result. 

\begin{prop} \label{compactembed}
Consider any $0< \beta < \alpha < 1$ and $\gamma > 0$. 
Then the following inclusion is compact
   \begin{equation}
       \iota:C^{k,\alpha}(M) \hookrightarrow x^{-\gamma}C^{k,\beta}(M).
   \end{equation}
\end{prop}

\begin{proof}
    Let $\{u_{n}\}_{n}$ be a bounded sequence of functions in $C^{k,\alpha}(M)$ and, 
    for any $\delta >0$, let $M_{\delta}$ be the compact submanifold given by
    \begin{equation}
        M_{\delta} = M \, \backslash \, \setdef{p \in M}{x(p)<\delta}.
    \end{equation}
    We know that $C^{k,\alpha}(M_{\delta})\hookrightarrow C^{k,\beta}(M_{\delta})$ compactly 
    for any $\delta>0$. Therefore, $\{u_{n}|_{M_{\delta}}\}_{n}$ admits a subsequence 
    $\{u_{n_j(\delta)}|_{M_{\delta}}\}_{j}$ which converges in $C^{k,\beta}(M_{\delta})$. Now consider a 
    sequence $\delta_i := 1/i$ for $i \in \N$. We define convergent subsequences in $C^{k,\beta}(M_{\delta_i})$
    for any $i$ by an iterative procedure: given a convergent subsequence
    $\{u_{n_j(\delta_i)}|_{M_{\delta_i}}\}_{j} \subset C^{k,\beta}(M_{\delta_i})$,
    we choose a convergent subsequence $\{u_{n_j(\delta_{i+1})}|_{M_{\delta_{i+1}}}\}_{j} \subset C^{k,\beta}(M_{\delta_{i+1}})$
    from $\{u_{n_j(\delta_i)}|_{M_{\delta_{i+1}}}\}_{j}$. Define the diagonal sequence by
    \begin{equation}
        \{v_{j} := u_{n_j(\delta_j)}\}_{\, j}.
    \end{equation}
    We claim that $\{v_{j}\}_{j}$ is a Cauchy sequence in $x^{-\gamma}C^{k,\beta}(M).$  In fact
    \begin{align*}
        \|v_{j}\|_{x^{-\gamma}C^{k,\beta}(M\backslash M_{\delta_j})} 
        = \|x^{\gamma}v_{j}\|_{C^{k,\beta}(M\backslash M_{\delta_j})} 
        \leq C\delta_{j}^{\gamma}, 
    \end{align*}
    where $C>0$ is an upper bound for the norms of $\{u_{n}\}_{n} \subset C^{k,\beta}(M)$.
    Now, let $\varepsilon>0$ and choose $j_{0} \in \mathbb{N}$ sufficient large such that $C\delta_{j_{0}}^{\gamma} \leq \varepsilon/4$.  
    The sequence $\{v_j|_{M_{\delta_{j_0}}}\} \subset C^{k,\beta}(M_{\delta_{j_0}})$ converges by construction and thus converges 
    also in $x^{-\gamma}C^{k,\beta}(M_{\delta_{j_0}})$.  Hence, there exists some $N_{0} \in \mathbb{N}$ sufficiently large, such that 
    for every $j,j' \geq N_{0}$
    \begin{equation} \label{ineq1}
        \|v_j - v_{j'}\|_{x^{-\gamma}C^{k,\beta}(M_{\delta_{j_0}})} \leq \varepsilon/2,
    \end{equation}
Hence for $J_{0} = \max\{j_{0},N_{0}\}$, we have for any $j,j' \geq J_{0}$
    \begin{align*}
        \|v_{j}-v_{s}\|_{x^{-\gamma}C^{k,\beta}(M)} &\leq \|v_{j}-v_{s}\|_{x^{-\gamma}C^{k,\beta}(M_{\delta_{j_0}})} 
        + \|v_{j}-v_{s}\|_{x^{-\gamma}C^{k,\beta}(M \backslash M_{\delta_{j_0}})} \\
        &<\varepsilon/2 + 2 \varepsilon/2 = \varepsilon.
    \end{align*}
    Hence, $\{v_{j}\}$ is a Cauchy sequence in $x^{-\gamma}C^{k,\beta}(M)$ and
    by completeness, it admits a convergent subsequence. This proves the statement.
\end{proof}

\noindent We can finally prove convergence of the $\mbox{CYF}^{+}$ flow, i.e. Theorem \ref{convergence}. 

\begin{thm} \label{convergencelong}
Let $(M,g_{0})$ be a manifold with bounded geometry such that $\scal(g_{0}) \in C^{4,\alpha}(M)$ is negative and bounded away from zero.  
Consider the global solution $u \in C^{4,\alpha}(M\times \mathbb{R}_{+})$ of $\mbox{CYF}^{+}$.  Then the family of metrics $\{g(t) = u(t)^{1/\eta}g_{0}\}_{t\geq 0}$ converges to a metric $g^{*} = (u^{*})^{1/\eta}g_{0}$ with constant negative scalar curvature.
\end{thm}

\begin{proof} 
From Proposition \ref{uniformboundpartialu}, $\| \partial_t u(t)\|_\infty$ decreases exponentially. From there it is 
easy to check that $u(t) \in L^\infty(M)$ is a Cauchy sequence and hence admits a well-defined limit $u^* \in L^\infty(M)$. 
By Proposition \ref{compactembed}, $u(t) \in C^{4,\alpha}(M)$ admits a convergent subsequence in $x^{-\gamma}C^{4,\alpha}(M)$
for any $\beta< \alpha$ and $\gamma>0$. Hence $u^* \in x^{-\gamma}C^{4,\alpha}(M)$ with
scalar curvature $\scal^* \in x^{-\gamma}C^{2,\alpha}(M)$ such that for some divergent sequence $(t_n) \in \R_+$
   \begin{equation}
       \|\scal g(t_n)-\scal^*\|_{x^{-\gamma}C^{2,\alpha}(M)} \rightarrow 0 \;\; \mbox{for} \; n\rightarrow \infty.
   \end{equation}
In particular, $\scal g(t_n)$ converges pointwise to $\scal^*$. 
Note that by Lemma \ref{evolutionscalar} the supremum $\sup_M \scal g(t)$ is non-increasing and by \eqref{minevolut}
the infimum $\inf_M \scal g(t)$ is non-decreasing. Thus $\sup_M \scal g(t)$ and $\inf_M \scal g(t)$ are bounded from 
below and above, respectively, and thus both convergent as $t\to \infty$. By Lemma \ref{exponentialapproxim}
$$
\lim_{t\to \infty} \sup_M \scal g(t) = \lim_{t\to \infty} \inf_M \scal g(t) =: \textup{const}.
$$
We compute from pointwise convergence of 
$\scal g(t)$ to $\scal^*$ at any $p \in M$
\begin{equation}\label{limit-scal-upper}
\begin{split}
\scal^*(p) &= \lim_{n\to \infty} \scal g(t_n)(p) \leq \lim_{n\to \infty} \sup_M \scal g(t_n) 
\\ &\Rightarrow \sup_M \scal^* \leq \textup{const}.
\end{split} 
\end{equation}
Similar argument applied to the infimum of $\scal^*$ yields
\begin{equation}\label{limit-scal-lower}
\begin{split}
\scal^*(p) &= \lim_{n\to \infty} \scal g(t_n)(p) \geq \lim_{n\to \infty} \inf_M \scal g(t_n) 
\\ &\Rightarrow \inf_M \scal^* \geq \textup{const}.
\end{split} 
\end{equation}
Combining \eqref{limit-scal-upper} and \eqref{limit-scal-lower}, proves the statement. 
\end{proof}

%%%%%%%%%%%%%%%%%%%%%%%%%%%%%
\section{Appendix: Yamabe flow on $\Phi$-manifolds}
%%%%%%%%%%%%%%%%%%%%%%%%%%%%

In this appendix we study short time existence of Yamabe flow in the specific class of $\Phi$-manifolds.  
There, we employ microlocal arguments to deduce stronger regularity statements for the flow. 
These techniques do not hold for a general manifold with bounded geometry.

\subsection{Fibered boundary manifolds}

Let $\overline{M} = M\cup \partial M$ be a compact $m$-dimensional smooth manifold whose boundary 
$\partial M$ is the total space of a fibration over a closed manifold $Y$ with typical fiber given by a 
closed manifold $Z$. We write $b:= \dim Y$ and $f:= \dim Z$ for the dimensions of the respective manifolds. 
We denote the fibration by
\begin{equation} \label{submersion}
    \phi: \partial M \rightarrow Y,
\end{equation}
which is a smooth surjective map such that $\phi^{-1}(\{y\}) =: Z_{y} \simeq Z$ for all $y \in Y$. 
Consider a collar neighborhood $\mathscr{U} \simeq [0,1)\times \partial M$ of the boundary $\partial M$
with a smooth boundary defining function $x:\mathscr{U} \to [0,1)$, i.e. $x^{-1}(\{0\}) = \partial M$ and $dx \neq 0$
at $\partial M$. We extend $x$ to a smooth nowhere vanishing function on $M$. 
We can now introduce a $\Phi$-metric in the open interior $M$.

\begin{defin} \label{phimanifolds}
A Riemannian metric $g_{\Phi}$ in the interior $M \subset \overline{M}$ of a fibered boundary manifold is 
said to be a $\Phi$-metric if in the collar $\mathscr{U}$ it can be written as
\begin{equation} 
g_{\Phi} = \dfrac{\di x^{2}}{x^{4}} + \dfrac{\phi^{*}g_{Y}}{x^{2}} + g_{Z} + h := g_{\Phi,0} + h,
\end{equation}
where $g_{Y}$ is a Riemannian metric on the base $Y$, $g_{Z}$ is a symmetric bilinear form on $\partial M$ 
which restricts to a Riemannian metric at each fiber $Z_{y}$, and the (higher order) term $h$ satisfies $|h|_{g_{\Phi,0}} = O(x)$ 
as $x\rightarrow 0$. We assume that 
$\phi: (\partial M, g_Z + \phi^{*}g_{Y}) \to (Z,g_Z)$ is a Riemannian submersion and 
call the Riemannian manifold $(M,g_{\Phi})$ a $\Phi$-manifold.
\end{defin}

Some examples of $\Phi$-metrics include the product of locally Euclidean metrics with a closed manifold.  Moreover, complete Ricci flat metrics are often $\Phi$-metrics and, furthermore, common classes of gravitational instatons are $\Phi$-metrics as well.  

Elliptic theory for $\Phi$-manifolds was first studied by Mazzeo and Melrose \cite{mazzeo1998pseudodifferential}, whose work was later extended by Grieser and Hunsicker \cite{grieserhuns1,grieserhuns2}.  Moreover, the development of Hodge theory on $\Phi$-manifolds is due to Hausel, Hunsicker and Mazzeo \cite{hhm} and Leichtnam, Mazzeo and Piazza \cite{lmp} have obtained results for Index theory on $\Phi$-manifolds.  More recently, one should mention the work by Talebi and Vertman \cite{VerTal} which provides a description of the asymptotic behavior of the heat kernel on $\Phi$-manifolds.

\subsection{Geometry-adapted H\"older spaces}\label{section-hoelder-def}
We define the space $\mathcal{V}_{\Phi}$ of $\Phi$-vector fields as the space of smooth vector fields that are bounded under $g_\Phi$.
Choose local coordinates $(x,y,z)$ on $\mathscr{U}$, where $(y)$ restricts to local coordinates on $\partial M$, lifted from $Y$ and $(z)$ restricts to local coordinates on each fibre $Z$. Then, locally, $\mathcal{V}_{\Phi}$ can be written as
\begin{equation} \label{phivectors}
    \mathcal{V}_{\Phi} = C^\infty(\overline{M}) - \Span \, \{x^{2}\partial_{x},x\partial_{y_{1}},...,
    x\partial_{y_{b}},\partial_{z_{1}},...,\partial_{z_{f}}\}.
\end{equation}
The universal enveloping algebra of $\mathcal{V}_{\Phi}$ is the ring $\mbox{Diff}^{*}_{\Phi}(M)$ 
of differential $\Phi$-operators. We denote by $\mathcal{V}^{l}_{\Phi}$ a set of generators for
$\mbox{Diff}^{l}_{\Phi}(M)$. By doing so, we are able to define the class of $k$-continuously $\R$-valued 
$\Phi$-differentiable functions 
\begin{equation}\begin{split}
     C^{k}_{\Phi}(M\times [0,T]) = 
    \biggsetdef{ u \in C^{0}(M\times [0,T])} 
        {\begin{array}{l}  (V \circ \partial^{l_{2}}_{t})u \in C^{\alpha}_{\Phi}(M\times [0,T]),
           \\
        \mbox{for} \; V \in \mbox{Diff}^{l_1}_{\Phi}(M), \hspace{2mm} l_{1} + 2l_{2} \leq k 
        \end{array}}.
\end{split}\end{equation}
Naturally, we can also define the class of geometry-adapted H\"older spaces by considering the class of $\Phi$-derivatives instead of regular derivatives and, moreover, by considering the distance function induced by $g_{\Phi}$, which is locally given by the following expression:
\begin{equation*} 
    d((x,y,z),(x',y',z')) = \sqrt{\left(\frac{x-x'}{(x+x')^{2}}\right)^{2}+
    \left(\frac{\|y-y'\|}{(x+x')} \right)^{2} + \|z-z'\|^2}.
\end{equation*}
Hence, we define
\begin{equation*}
    C^{k,\alpha}_{\Phi}(M\times [0,T]) = 
    \biggsetdef{ u \in C^{k}_{\Phi}(M\times [0,T])} 
        {\begin{array}{l}  (V \circ \partial^{l_{2}}_{t})u \in C^{\alpha}_{\Phi}(M\times [0,T]),
           \\
        \mbox{for} \; V \in \mbox{Diff}^{l_1}_{\Phi}(M), \hspace{2mm} l_{1} + 2l_{2} \leq k 
        \end{array}}
\end{equation*}
which is a Banach space, cf. \cite[Proposition 3.1]{bahuaud2014yamabe}, with the norm
\begin{equation}
\|u\|_{k,\alpha} := \|u\|_{\alpha} +\displaystyle\sum_{l_{1} + 2l_{2} \leq k} 
\sum_{V \in \mathcal{V}^{l_{1}}_{\Phi}} \|(V\circ \partial_{t}^{l_{2}})u\|_{\alpha}.
\end{equation}

%%%%%%%%%%%%%%%%%%%%%%%%%%%%%%%%%%%%%%%%%%%%%%%
\subsection{Conformal transformation by H\"{o}lder functions} 
%%%%%%%%%%%%%%%%%%%%%%%%%%%%%%%%%%%%%%%%%%%%%%%%

Since the Yamabe flow preserves the conformal class of the metric, we need to look into 
the effect of conformal transformation by H\"older functions. We first define
the conformal class of a $\Phi$-metric $g_\Phi$ (we tacitly assume $m\geq 3$)
\begin{equation} \label{conformalclass}
    [g_{\Phi}] = \bigsetdef{u^{4/(m-2)}\cdot g_{\Phi}}{u \in C^{2}_{\Phi}(M), \;\; \displaystyle\inf_{M}u > 0, \;\;\|u\|_{\infty}<+\infty}.
\end{equation}
First, observe that a generic element of the conformal class $ [g_{\Phi}] $ is \emph{not}
a $\Phi$-metric in the sense of Definition \ref{phimanifolds}, since the conformal 
factor $u^{4/(m-2)}$ cannot, in general, be expected to admit a partial asymptotic 
expansion as $x\to 0$. However, it still has $\mathcal{V}_\Phi$ as the space of bounded 
vector fields and thus the distance functions defined with respect to any $g \in [g_{\Phi}]$
are equivalent. In that sense $g$ still has the same $\Phi$-geometry as $g_\Phi$ and we conclude

\begin{prop}
 The H\"older spaces defined in \S \ref{section-hoelder-def} 
 do not depend on the choice of a metric $g \in [g_{\Phi}]$.
\end{prop}

%%%%%%%%%%%%%%%%%%%%%%%%%%%%%%%%%%%%%%%%%%%%%%%%%%%%
\subsection{Scalar curvature of $(M,g_{\Phi})$}
%%%%%%%%%%%%%%%%%%%%%%%%%%%%%%%%%%%%%%%%%%%%%%%%%%%

In this work we assume that the scalar curvature $\scal(g_{\Phi})$ of $(M,g_{\Phi})$ is negative and bounded 
uniformly away from zero. In order to understand the geometric restrictions it entails, we should consider 
the asymptotic expansion of $\scal(g_{\Phi})$ near the boundary $\partial M$. Employing \cite[Chapter 7, Corollary 43]{oneill} 
and the tensor properties of the scalar curvature, we obtain for the trivial fibration and vanishing higher order term $h$
\begin{equation}\label{curvature}
    \scal(g_{\Phi,0}) = x^{2}(\scal(g_{Y}) + b(b-1)) + \scal(g_{Z}),
\end{equation}
which reflects the behavior of $\scal(g_{\Phi})$ near $x=0$. In the general case, additional $O(x)$ terms (as $x\to 0$) appear.

%%%%%%%%%%%%%%%%%%%%%%%%%%%%%%%%%%%%%%%%%%%%%%%%%%%%
\subsection{Bounded geometry of $(M,g_{\Phi})$}
%%%%%%%%%%%%%%%%%%%%%%%%%%%%%%%%%%%%%%%%%%%%%%%%%%%

We note that $\Phi$-manifolds are a particular case of manifold with bounded geometry.  Therefore, all results obtained in this work hold, in particular, on $(M,g_{\Phi})$.  In fact, from \cite[Chapter 7, Corollary 43]{oneill} and omitting the indexes, it follows that
\begin{equation}
    \ric_{\Phi,0} = \Bigl(\ric_{Y} + (b-1)g_{Y}\Bigr) + \ric_{Z},
\end{equation}
which reflects the behavior of $\ric_{\Phi}$ near $x = 0$.  Moreover, the sectional curvature $K_{M}$ is also bounded near $\partial M$ and thus, from the work of Cheeger, Gromov and Taylor \cite[Theorem 4.7]{cgt}, it is possible to check that $r_{\mbox{inj}}(M)$ is positive and bounded from below away from zero.  Thus, we conclude $\Phi$-manifolds have bounded geometry.

%%%%%%%%%%%%%%%%%%%%%%%%%%%%%%%%%%%%%%%%%%%%%%%
\subsection{Asymptotics of the heat kernel}
%%%%%%%%%%%%%%%%%%%%%%%%%%%%%%%%%%%%%%%%%%%%%%%%

In order to derive mapping properties of the heat kernel $H$, we need to recall briefly the asymptotic structure of the heat kernel. We
refer the reader to \cite{VerTal} and \cite{giuseppebruno} for further details. Specifically, the heat kernel $H$ is 
a smooth function in the open interior of $\overline{M}^2 \times [0,\infty)_t$ with 
singular behavior at 
\begin{align*}
&\textup{FF}:=\partial M \times \partial M \times [0,\infty), \\
&\textup{FD}:=\{y=y'\} \times [0,\infty) \subset \textup{FF}, \\
&\textup{TD}:=\textup{diag} (\partial M \times \partial M) \times \{t=0\}.
\end{align*} 
This singular behavior is resolved by blowing up these singular submanifolds, 
i.e. replacing the submanifolds by their inward pointing normal bundles, glued 
into $M^2 \times (0,\infty)$ in a well-defined geometric way. 
The inward pointing normal bundles of (the lifts of) FF, FD, TD are then 
new boundary faces in the blowup space $\mathscr{M}^2_h$, referred to as
ff, fd and td, respectively. The blowup space $\mathscr{M}^2_h$ is illustrated in Figure \ref{thirdblowup}.  The regions in $\mathscr{M}^{2}_h$ identified below with numbers ranging from 1 to 5 are called ``regimes''.
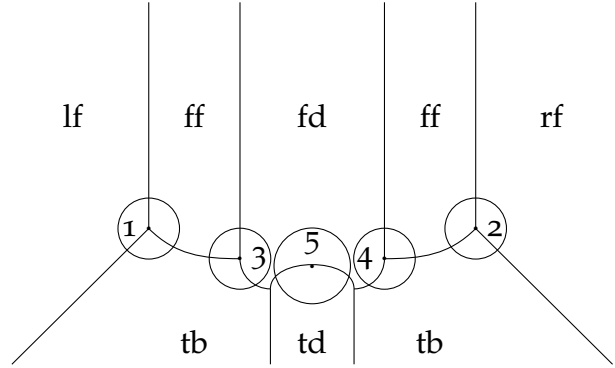
\begin{figure}[!htbp]
    \centering
    \begin{tikzpicture}
\draw (0,0) node[left]{1} to (0,3);
\draw (0,0) to (-1.8,-1.8);
\draw (0,0) node[circle,inner sep=7pt,draw]{.} to[out=315,in=180] (1.2,-0.4);
\draw (1.2,-0.4) node[right]{3} to (1.2,3);
\draw (1.2,-0.4) node[circle,inner sep=7pt,draw]{.} to[out=270,in=180] (1.6,-0.8);
\draw (1.6,-0.8) to (1.6,-1.8);
\draw (1.6,-0.8) to[out=90,in=90] (2.7,-0.8);
\draw (2.7,-0.8) to (2.7,-1.8);
\draw (2.7,-0.8) to[out=0,in=270] (3.1,-0.4);
\draw (3.1,-0.4) node[left]{4} to (3.1,3);
\draw (3.1,-0.4) node[circle,inner sep=7pt,draw]{.} to[out=0,in=225] (4.3,0);
\draw (4.3,0) node[circle,inner sep=7pt,draw]{.} to (4.3,3);
\draw (4.3,0) node[right]{2} to (6.1,-1.8);
\draw (2.15,-0.495) node[circle,inner sep=9pt,draw]{.};
\draw (2.15,-0.495) node[above]{5};
\draw (-1,1.5) node{lf};
\draw (0.6,1.5) node{ff};
\draw (2.15,1.5) node{fd};
\draw (3.7,1.5) node{ff};
\draw (5.3,1.5) node{rf};
\draw (0.6,-1.5) node{tb};
\draw (2.15,-1.5) node{td};
\draw (3.7,-1.5) node{tb};
\end{tikzpicture}
    \caption{The parabolic blowup space $\mathscr{M}^2_h$.}
    \label{thirdblowup}
\end{figure}

Local coordinates on the blowup space are best understood in terms of projective coordinates, written in terms of 
local coordinates $(x,y,z)$ and $(\widetilde{x}, \widetilde{y}, \widetilde{z})$ on the two copies of $M$. 
In the regime $1$ (coordinates in the regime $2$ are obtained by interchanging the roles of $x$ and $\widetilde{x}$)
we have the projective coordinates
\begin{equation}\label{lowerleftcorner}
    (x,y,z,\widetilde{s},\widetilde{y},\widetilde{z},\tau) := \left(x,y,z,\dfrac{\widetilde{x}}{x},\widetilde{y},\widetilde{z},\sqrt{t}\right).
\end{equation}
In these coordinates, the defining functions $\rho_{\lf}, \rho_{\ff}, \rho_{\tb}$ of the boundary faces
lf, ff and tb, are given by $\widetilde{s}, x$ and $\tau$, respectively. \medskip

In the regime $3$ (coordinates in the regime $4$ are obtained by interchanging the roles of $x$ and $\widetilde{x}$) 
we have the projective coordinates
\begin{equation}\label{coordinatesnearfd}
    \left(x,y,z,\mathcal{S}',\mathcal{U}',\mathcal{Z}',\tau\right):=
     \left(x,y,z,\dfrac{\widetilde{s}-1}{x},\dfrac{\widetilde{y}-y}{x},\widetilde{z}-z, \sqrt{t}\right).
\end{equation}
In these coordinates, the defining functions $\rho_{\ff}, \rho_{\fd}, \rho_{\tb}$ of the boundary faces
ff, fd and tb, are given by $\|(\mathcal{S}',\mathcal{U}')\|^{-1}, x$ and $\tau$, respectively. \medskip

The projective coordinates in the regime 5 are given by
\begin{equation}\label{coordinatesneartd}
    (x,y,z,\mathcal{S},\mathcal{U},\mathcal{Z}, \tau) 
    := \left(x,y,z,\dfrac{\mathcal{S}'}{\tau},\dfrac{\mathcal{U}'}{\tau},\dfrac{\mathcal{Z}'}{\tau}, \sqrt{t} \right). 
\end{equation}
In these coordinates, the defining functions $\rho_{\fd}, \rho_{\td}$ of the boundary faces
fd and td, are given by $x$ and $\tau$, respectively, while $\|(\mathcal{S},\mathcal{U},\mathcal{Z})\|\rightarrow \infty$ 
corresponds to $\tb$. Lifting $H$ to $\mathscr{M}^2_h$ corresponds in local coordinates simply 
to a change to projective coordinates \eqref{lowerleftcorner}, \eqref{coordinatesnearfd}
or \eqref{coordinatesneartd}. Now we can state the asymptotics of the heat kernel $H$.

\begin{thm}\cite[Theorem 7.2]{VerTal}\label{MohammadThm}
Let $(M,g_{\Phi})$ be an $m$-dimensional $\Phi$-manifold. Then the heat kernel $H$ 
lifts to a polyhomogeneous function $\beta^*H$ on the blowup space $\mathscr{M}^2_h$ 
with the following asymptotic behavior
\begin{equation}
\beta^*H \sim \rho_{\lf}^\infty \rho_{\ff}^\infty\rho_{\rf}^\infty\rho_{\tb}^\infty\rho_{\fd}^0\rho_{td}^{-m}G_{0}
\end{equation}
with $G_{0}$ being a bounded function.
\end{thm} 

%%%%%%%%%%%%%%%%%%%%%%%%%%%%%%%%%%%%%%%%%%%%%%%
\subsection{Mapping properties of the heat operator}
%%%%%%%%%%%%%%%%%%%%%%%%%%%%%%%%%%%%%%%%%%%%%%%%

The mapping properties for the heat kernel $H$ proved in \cite[Theorem 1.1]{giuseppebruno} are stronger than those presented here, since there the authors actually study H\"older spaces
with respect to the distance of the incomplete metric $x^4g_\Phi$. The mapping properties here
are still sufficient for our purposes, and the proof follows along the lines of \cite{giuseppebruno}.

\begin{prop}\label{modified-mapping} 
The heat operator $H$ acting by convolution in time defines, for 
any $k\in \N_0$ and $\alpha \in (0,1)$, bounded linear mappings
\begin{equation}\begin{split}
&H:x^{\gamma}C^{k,\alpha}_{\Phi}(M\times [0,T]) \rightarrow x^{\gamma}C^{k+2,\alpha}_{\Phi}(M\times [0,T]), \\
&H:x^{\gamma}C^{k+1,\alpha}_{\Phi}(M\times [0,T]) \rightarrow \sqrt{t} \ x^{\gamma}C^{k+2,\alpha}_{\Phi}(M\times [0,T]).
\end{split}
\end{equation}
\end{prop}

\begin{proof} Note that the mapping properties above are equivalent to
\begin{equation}
    H_{\gamma}:= x^{-\gamma}Hx^{\gamma}:C^{k,\alpha}_{\Phi}(M\times [0,T])\rightarrow \Bigl(C^{k+2,\alpha}_{\Phi}\cap\;\sqrt{t}\;C^{k,\alpha}_{\Phi}\Bigr)(M\times [0,T])
\end{equation}
acting continuously.  Moreover, the kernel of $H_{\gamma}$ has the same asymptotic behavior of the heat kernel $H$ near the boundary hypersurfaces, which means we can work directly with $H_{\gamma}$.  We first discuss the proof for $k=0$.
The statement is then equivalent to boundedness of 
\begin{equation}\label{V}
    G = V \circ H_{\gamma}: 
   C^{\alpha}_{\Phi}(M\times [0,T])\rightarrow C^{\alpha}_{\Phi}(M\times [0,T]),
\end{equation}
where $V \in \mbox{Diff}^{l}_{\Phi}(M)$, for $l \in \{0,1,2\}$.  Moreover, it should be noted that, for estimates purposes, it is possible to consider $V$ to be simply a generator of $\mbox{Diff}^{l}_{\Phi}(M)$.
The key is to use the (local) H\"older norm in \eqref{eq.local-norm}, which means that the H\"older differences should be estimated only for
$$d(p,p')^{\alpha} + |t-t'|^{\alpha/2} \leq \delta.$$
Boundedness of $G$ is then established in the three following steps:
\medskip

 \begin{enumerate}
        \item Estimates for spacial difference: if $d(p,p')^{\alpha} \leq \delta$, then
        $$|G u(p,t) - G u(p',t)| \leq C \|u\|_{\alpha}d(p,p')^{\alpha},$$
        \item Estimates for time difference: $$|G u(p,t) - G u(p,t')| \leq C \|u\|_{\alpha}|t-t'|^{\alpha/2},$$
        \item Estimates for supremum norm: $$|G u(p,t)| \leq C \|u\|_{\alpha},$$
    \end{enumerate} for some uniform constants $C>0$ independent of $u$ and $(p,p',t,t')$.
    In fact, we will denote uniform positive constants always by $C$ and $c$, despite the 
    constants possibly being different from estimate to estimate. 
    \medskip

\noindent \textbf{Proof of (i):} Consider $p,p' \in M$ and write
\begin{align*}
    &M^{+}=\setdef{\tilde{p} \in M}{ \;d(p,\tilde{p})\leq 3d(p,p')}, \\
    &M^{-}=\setdef{\tilde{p} \in M}{\;d(p,\tilde{p}) \geq 3d(p,p')}.
\end{align*}
We shall assume that $p=(x,y,z)$ and $p'=(x',y,z)$, with $x'> x$ without loss of generality. The cases where $p$ and $p'$ differ in the $(y,z)$
components, are discussed similarly. We write out the estimate in the regime $5$ in Figure \ref{thirdblowup}, where 
fd meets td. The other regimes are simpler. \medskip

Below we use the mean value theorem with $p_\xi = (\xi, y, z)$ 
for some intermediate $\xi \in (x,x')$.  Moreover, stochastic completeness of $\Phi$-manifolds implies that,
by replacing $u(\tilde{p},\tilde{t})$ by $(u(\tilde{p},\tilde{t}) - u(p,\tilde{t}))$
 (writing $\tilde{p}= (\widetilde{x}, \widetilde{y}, \widetilde{z})$), we can write
\begin{align*}
        Gu(p,t) - &Gu(p',t) = \\ (x-x') &\int_{0}^{t}\int_{M^{-}}\partial_\xi G(t-\tilde{t},p_\xi,\tilde{p})
        \bigl( u(\tilde{p},\tilde{t}) - u(p,\tilde{t}) \bigr) 
        \dvol_{\Phi}(\tilde{p})\di \tilde{t} \\ 
         +\, &\int_{0}^{t}\int_{M^{+}}G(t-\tilde{t},p,\tilde{p})
         \bigl( u(\tilde{p},\tilde{t}) - u(p,\tilde{t}) \bigr) 
         \dvol_{\Phi}(\tilde{p})\di \tilde{t} \\
         -\, &\int_{0}^{t}\int_{M^{+}}G(t-\tilde{t},p',\tilde{p})
         \bigl( u(\tilde{p},\tilde{t}) - u(p',\tilde{t}) \bigr) 
         \dvol_{\Phi}(\tilde{p})\di \tilde{t} \\
       +\, &\int_{0}^{t}\int_{M^{+}}G(t-\tilde{t},p',\tilde{p})
         \bigl( u(p,\tilde{t}) - u(p',\tilde{t}) \bigr) 
         \dvol_{\Phi}(\tilde{p})\di \tilde{t} \\
       =&I_{1}+I_{2}-I_{3} +I_4,
\end{align*}

%%%%%%%%%%%%%%%%%%%%%%%%%%%%%%%%%%%
\subsection{Estimates for $I_1$}
%%%%%%%%%%%%%%%%%%%%%%%%%%%%%%%%%%%

From the H\"older continuity of $u$, we write 
\begin{align*}
|I_1| &\leq |x-x'| \cdot \|u\|_\alpha \int_{0}^{t}\int_{M^{-}}\partial_\xi G(t-\tilde{t},p_\xi,\tilde{p})
        \cdot d(\tilde{p},p)^\alpha \dvol_{\Phi}(\tilde{p})\di \tilde{t} \\
        &\leq |x-x'| \cdot  \|u\|_\alpha \int_{0}^{t}\int_{M^{-}}\partial_\xi G(t-\tilde{t},p_\xi,\tilde{p})
        \cdot d(\tilde{p},p_\xi)^\alpha \dvol_{\Phi}(\tilde{p})\di \tilde{t},
\end{align*}
where in the second estimate we used $d(\tilde{p},p) \leq 3d(\tilde{p},p_\xi)$,
obtained by exactly the same arguments as in \cite[(6.1)]{giuseppebruno}.
Since we estimate in the regime $5$ in Figure \ref{thirdblowup}, where 
fd meets td, we use the local projective coordinates
$(\tau,\xi,y,z,\mathcal{S}',\mathcal{U}',\mathcal{Z}')$, introduced in \eqref{coordinatesnearfd}, where 
\[\mathcal{S}' = \dfrac{\widetilde{x}-\xi}{\xi^{2}}, \hspace{2mm} \mathcal{U}' = \dfrac{\widetilde{y}-y}{\xi}, \hspace{2mm} \mathcal{Z}' = \widetilde{z}-z \hspace{2mm} \mbox{and} \hspace{2mm} \tau = \sqrt{t - \widetilde{t}}.\]
Then we compute from Theorem \ref{MohammadThm} and $\dvol_\Phi(\widetilde{x},\widetilde{y},\widetilde{z}) \sim 
\widetilde{x}^{-2-b} \di\widetilde{x} \di\widetilde{y} \di\widetilde{z}$
\begin{align*}
|I_1|\le c \cdot \frac{|x-x'|}{\xi^2} \cdot \|u\|_\alpha  \int_{0}^{\sqrt{t}}\int_{M^-}\tau^{-m-2} G_{0}
 \sqrt{ |\mathcal{S}'|^2 + \|\mathcal{U}'\|^2 + \|\mathcal{Z}'\|^2 }^{\, \alpha} \\ \di\mathcal{S}'\di\mathcal{U}'\di\mathcal{Z}'\di\tau,
\end{align*}
with $G_{0}$ being bounded and vanishing to infinite order as $\|(\mathcal{S},\mathcal{U},\mathcal{Z})\| \to \infty$,
where $(\mathcal{S},\mathcal{U},\mathcal{Z}) = (\mathcal{S}'/\tau,\mathcal{U}'/\tau,\mathcal{Z}'/\tau)$.
Let us define
$$r(\mathcal{S}',\mathcal{U}',\mathcal{Z}') := \sqrt{ |\mathcal{S}'|^2 + \|\mathcal{U}'\|^2 + \|\mathcal{Z}'\|^2 }.$$
Such a function $r$ describes the radial distance in polar coordinates around the origin.  
Performing a change of coordinates, we obtain
$$|I_1|\le c \cdot \frac{|x-x'|}{\xi^2} \cdot \|u\|_\alpha 
\int_{0}^{\sqrt{t}}\int_{M^{-}} \tau^{-m-2}r^{m-1+\alpha}G_{0}\di r\di \mbox{(angle)}\di\tau. $$
Now, setting $\sigma=r^{-1}\tau =  \sqrt{ |\mathcal{S}|^2 + \|\mathcal{U}\|^2 + \|\mathcal{Z}\|^2 }^{-1}$, it follows that $G_{0}$ against any negative power of $\sigma$ is bounded. 
Hence, integrating out the angular variables, followed by another change of coordinates $\tau\mapsto\sigma$ gives
\begin{align*}
|I_1| &\le c \cdot \frac{|x-x'|}{\xi^2} \cdot \|u\|_\alpha \int_{0}^{\sqrt{t}}\int_{M^{-}} r^{-2+\alpha}\di r.
\end{align*}
Now, exactly as in \cite[(6.3)]{giuseppebruno} we find $M^{-} \subset \{d(p,p') \le cr\}$ for some constant $c>0$. 
Thus we can estimate even further
\begin{equation}\label{I1}\begin{split}
|I_1|&\le c \cdot \frac{|x-x'|}{\xi^2} \cdot \|u\|_\alpha \int_{c\frac{d(p,p')}{\xi^2}}^\infty r^{-2+\alpha}\di r\\
&=c \cdot \frac{|x-x'|}{\xi^2} \cdot d(p,p')^{-1+\alpha} \|u\|_\alpha.
\end{split}\end{equation}
In order to conclude the desired estimate of $I_1$, 
recall from Lemma \ref{local-norm}, that we may consider only $d(p,p') \leq \delta^{1/\alpha} =: \rho$, 
with any positive $\rho < 1/4$. Then $$1-\frac{x}{x'} \leq 2 \rho (x+x') \leq 4\rho.$$
Thus $x > (1-4\delta) x'$. Hence we may estimate
\begin{align*}
\frac{|x-x'|}{\xi^2} &\leq \frac{|x-x'|}{x^2} 
\leq (1-4\rho) ^{-2} \frac{|x-x'|}{x'^2} \\ &\leq 4(1-4\rho) ^{-2} \frac{|x-x'|}{(x+x')^2} 
\leq 4(1-4\rho) ^{-2} d(p,p').
\end{align*}
Thus for $\delta>0$ sufficiently small, we conclude from \eqref{I1} and the last estimate above
\begin{align*}
|I_1|\le c \cdot d(p,p')^{\alpha} \|u\|_\alpha.
\end{align*}
%%%%%%%%%%%%%%%%%%%%%%%%%%%%%%%%%%%
\subsection{Estimates for $I_2, I_3$}
%%%%%%%%%%%%%%%%%%%%%%%%%%%%%%%%%%%
Similar estimates as above lead to
\begin{align*}
|I_2|, |I_3| \le c \cdot \|u\|_\alpha \int_0^{c d(p,p')} r^{-1+\alpha}\di r
 \leq c \|u\|_\alpha d(p,p')^{\alpha},
\end{align*}
implying both estimates.

%%%%%%%%%%%%%%%%%%%%%%%%%%%%%%%%%%%
\subsection{Estimates for $I_4$}
%%%%%%%%%%%%%%%%%%%%%%%%%%%%%%%%%%%
For the estimate of $I_{4}$, we assume again as before that the heat kernel 
is supported near fd meeting td, and thus work with local projective coordinates 
$(\tau,x',y',z',\mathcal{S},\mathcal{U},\mathcal{Z})$ given in \eqref{coordinatesneartd}, that is,
\begin{equation*}
    \mathcal{S} = \dfrac{\widetilde{x}-x'}{x'^{2}\tau},\;\; \mathcal{U} = \dfrac{\widetilde{y}-y'}{x'\tau},\;\; 
    \mathcal{Z} = \dfrac{\widetilde{z}-z'}{\tau}\;\;\mbox{and}\;\; \tau = \sqrt{t-\widetilde{t}}.
\end{equation*}
We will obtain the estimates using integration by parts. To do so, 
note that one has (as the "worst case scenario" with $V \in \mbox{Diff}^{2}_{\Phi}(M)$) 
$G = \tau^{-m-2}(X_{1}X_{2}H)$ with both $X_{1},X_{2} \in 
\{\partial_{\mathcal{S}},\partial_{\mathcal{U}},\partial_{\mathcal{Z}}\}$. 
For the sake of simplicity, we shall assume $X_{1} = \partial_{\mathcal{S}}$. On the other hand, one has 
by triangle inequality
\begin{equation}\label{partialM}\begin{split}
\partial M^{+} &= \Bigl\{ d\bigl((x,y,z),(\widetilde{x},\widetilde{x},\widetilde{x})\bigr) = 3 d\bigl((x,y,z),(x',y',z')\bigr)\Bigr\} \\
&\subseteq \Bigl\{ 2 d\bigl((x,y,z),(x',y',z') \leq d\bigl((x',y',z'),(\widetilde{x},\widetilde{x},\widetilde{x})\bigr) \Bigr\}.
\end{split}\end{equation}
Moreover we can also write for some smooth function $\ell$
\[\beta^{*}(\dvol_{\Phi}(\widetilde{p})\di \widetilde{t}) = \ell \bigl(x'+\tau \, x'^{2}\mathcal{S},y'+\tau \, x' \mathcal{U},z'+\tau \, \mathcal{Z}\bigr) \di 
\mathcal{S} \di \mathcal{U} \di \mathcal{Z} \di \tau,\]
Since $u(p,\widetilde{t})- u(p',\widetilde{t}) =: \delta u$ is independent of $\widetilde{p}$, we can integrate by parts
\begin{align*}
I_4 &= \int_{0}^{\sqrt{t}}\delta u\int_{M^{+}} \tau^{-1}(\partial_{\mathcal{S}}X_{2} H) \ell \di \mathcal{S} \di \mathcal{U}\di \mathcal{Z} \di \tau\\
&= \int_{0}^{\sqrt{t}}\delta u\int_{\partial M^{+}} \tau^{-1}(X_{2} H) \ell \di \mathcal{U}\di \mathcal{Z} \di \tau\\
&\hspace{4mm} - \int_{0}^{\sqrt{t}}\delta u\int_{ M^{+}} \tau^{-1}(X_{2} H)\partial_{\mathcal{S}} \ell \di \mathcal{S} \di \mathcal{U}\di \mathcal{Z} \di \tau
=: I^{1}_4 - I^{2}_4.
\end{align*}
For the $I^{2}_4$-term, note that $\ell$ is a smooth function and therefore $\partial_{\mathcal{S}} \ell = \tau \, x'^{2} \partial_{\widetilde{x}} \ell$.
This cancels the $\tau^{-1}$ in the integrand and thus $I^{2}_4$ can be estimated against $\|u\|_\alpha d(p,p')^{\alpha}$. 
For the $I^{1}_4$-term, note by \eqref{partialM} that we can estimate
\begin{align*}
|I^1_4| \leq \|u\|_\alpha d(p,p')^{\alpha} \int_{\partial M^{+}} \tau^{-1} \bigl( X_{2}H \bigr)
\leq \frac{1}{2} \, \|u\|_\alpha \int_{\partial M^{+}} \tau^{-1} \bigl( X_{2}H \bigr) d(p',\widetilde{p})^{\alpha}
\end{align*}
This can now be estimated exactly as $I_2, I_3$,
completing the proof for (i). \bigskip
    
\noindent \textbf{Proof of (ii):} For time difference, first assume $t'<t$ (without loss of generality) and suppose first $t\leq 2t'$.  
Let us consider the case where $V$ in \eqref{V} is a first or second order $\Phi$-derivative, so that we can apply 
stochastic completeness. Then we find by the mean value theorem for some intermediate $\theta \in (t',t)$
\begin{align*}
    Gu(p,t) - &Gu(p,t') = \\ &|t-t'|\int_{T_{-}}\int_{M} \partial_{\theta}G(\theta-\tilde{t},p,\tilde{p})
    \bigl( u(\tilde{p},\tilde{t}) - u(p,\tilde{t}) \bigr) \dvol_{\Phi}(\tilde{p})\di \tilde{t} \vspace{1mm}\\
&+ \int_{T_{+}}\int_{M} G(t-\tilde{t},p,\tilde{p}) \bigl( u(\tilde{p},\tilde{t}) - u(p,\tilde{t}) \bigr)
\dvol_{\Phi}(\tilde{p})\di \tilde{t} \\
&- \int_{T'_{+}}\int_{M} G(t'-\tilde{t},p,\tilde{p})\bigl( u(\tilde{p},\tilde{t}) - u(p,\tilde{t}) \bigr)
\dvol_{\Phi}(\tilde{p})\di \tilde{t} \\
=: &L_{1} + L_{2} - L_{3},
    \end{align*}
with the subsets $T_{-}, T_{+}$ and $T'_{+}$ defined as follows:
\begin{equation}
    T_{-}=[0,2t'-t], \;\; T_{+}=[2t'-t,t] \;\; \mbox{and} \;\; T'_{+}=[2t'-t,t'].
\end{equation}
If $V$ is identity, the estimates follow similar to those of $L_1$ with $T_-$ replaced by $[0,T]$.
Using H\"older continuity of $u$ we obtain by Theorem \ref{MohammadThm} in 
projective coordinates \eqref{coordinatesneartd}
\begin{equation*}
    |L_{1}| \leq C |t-t'| \|u\|_\alpha \int_{T_{-}} \tau^{-3+\alpha}, \quad
    |L_{2}|, |L_{3}| \leq C \|u\|_\alpha \int_{T_{+}} \tau^{-1+\alpha}, 
\end{equation*}
where $\tau = \sqrt{\theta - \tilde{t}}$ in the first integral, and 
$\tau = \sqrt{t - \tilde{t}}$ in the second. Note that for $\tilde{t} \in T_-$ we have $(\theta - \tilde{t}) \geq (t-\tilde{t})$.
From there we conclude immediately the statement (ii). \bigskip

\noindent \textbf{Proof of (iii):} Let us consider the case where 
$V$ in \eqref{V} is a first or second order $\Phi$-derivative, so that we can apply 
stochastic completeness. Then 
\begin{align*}
    |Gu(p,t)| \leq \int_0^T \int_{M} G(t-\tilde{t},p,\tilde{p})| u(\tilde{p},\tilde{t}) - u(p,\tilde{t})| \dvol_{\Phi}(\tilde{p})\di \tilde{t}. 
\end{align*}
Using H\"older continuity of $u$ we obtain by Theorem \ref{MohammadThm} in 
projective coordinates \eqref{coordinatesneartd}
\begin{equation*}
    |Gu(p,t)| \leq  C \|u\|_\alpha  \int \tau^{-1+\alpha} \leq C \|u\|_\alpha.
\end{equation*}
If $V$ is identity, the estimate follows along the same lines without the stochastic completeness trick.  
This completes the proof of the statement for $k=0$. For general $k$, in all of the above integrals we can 
first pass $k$ $\Phi$-derivatives to the function $u$ using integration by parts in $(\mathcal{S},\mathcal{U}, \mathcal{Z})$
and then continue as before in case $k=0$.
\end{proof}

Once we proved the mapping properties in Proposition \ref{modified-mapping}, one can linearize the Yamabe flow equation \eqref{transformed2} -- as described in \S \ref{shorttime} -- by considering $u = 1 + v$, where $v \in x^{\gamma}C^{k,\alpha}_{\Phi}(M\times [0,T])$.  Moreover, since $(M,g_{\Phi})$ has bounded geometry, the mapping properties for the operators $F_{1}$ and $F_{2}$, from Lemmas \ref{lemmaf2} and \ref{lemmaf1}, still hold for weighted H\"older spaces and thus, the contraction argument can be used once again to prove short-time existence of the Yamabe flow on $\Phi$-manifolds with a linearizing factor lying on $x^{\gamma}C^{k,\alpha}_{\Phi}(M\times [0,T])$, which does not hold for a general manifold with bounded geometry.


\begin{thebibliography}{99}

\bibitem[\textsc{AlOw88}]{alvesowen}
\textsc{P.~Alves} and \textsc{R.~C. Owen}, \emph{Conformal deformation to
	constant negative scalar curvature on non-compact Riemannian manifolds},
Journal of Differential Geometry \textbf{27} (1988), 225--239.

\bibitem[\textsc{AMR16}]{aliasrigoli}
\textsc{L.~J. Al\'{i}as}, \textsc{P.~Mastrolia} and \textsc{M.~Rigoli},
\emph{Maximum principles and geometric applications}, Springer Monographs in Mathematics (2016).

\bibitem[\textsc{Aub76}]{Aubin}
\textsc{T.~Aubin}, \emph{\'{E}quations diff\'erentielles non lin\'eaires et
	probl\`eme de {Y}amabe concernant la courbure scalaire}, J. Math. Pures Appl.
(9) \textbf{55} (1976), no.~3, pp. 269--296.

\bibitem[\textsc{BaVe14}]{bahuaud2014yamabe}
\textsc{E.~Bahuaud} and \textsc{B.~Vertman}, \emph{Yamabe flow on manifolds
	with edges}, Mathematische Nachrichten \textbf{287} (2014), no.~2-3,
127--159.

\bibitem[\textsc{BaVe19}]{bahuaud2019long}
\bysame, \emph{Long-time existence of the edge yamabe flow}, Journal of the
Mathematical Society of Japan \textbf{71} (2019), no.~2, 651--688.


\bibitem[\textsc{Bre05}]{Brendle}
\textsc{S.~Brendle}, \emph{Convergence of the Yamabe Flow for Arbitrary Initial Energy}. 
Journal of Differential Geometry, 2005, Volume 69, Number 2, pp. 217--278.

\bibitem[\textsc{Bre07}]{BrendleYF}
\textsc{S.~Brendle}, \emph{Convergence of the {Y}amabe flow in dimension 6 and higher}, 
Invent. Math. \textbf{170} (2007), no.~3, pp. 541--576. \MR{2357502}

\bibitem[\textsc{CaGe22}]{giuseppebruno}
\textsc{B.~Caldeira} and \textsc{G.~Gentile}, \emph{Heat-type equations on manifolds with fibered boundary I: Schauder estimates}, preprint, arXiv:2203.15523 [math.AP] (2022).

\bibitem[\textsc{CGT82}]{cgt}
\textsc{J.~Cheeger}, \textsc{M.~Gromov} and \textsc{M.~Taylor}, \emph{Finite propagation speed, kernel estimates for functions of the Laplace operator, and the geometry of complete Riemannian manifolds}. J. Differential Geometry \textbf{17} (1982), no. 1, 15--53. MR0658471. 

\bibitem[GrHu09]{grieserhuns1}\textsc{D.~Grieser} and \textsc{E.~Hunsicker}, \emph{Pseudodifferential operator calculus for generalized $\mathbb Q$-rank 1 locally symmetric spaces}. I. J. Funct. Anal. \textbf{257} (2009), no. 12, 3748--3801. MR2557724.

\bibitem[GrHu09]{grieserhuns2}\textsc{D.~Grieser} and \textsc{E.~Hunsicker}, \emph{A parametrix construction for the Laplacian on $\mathbb Q$-rank 1 locally symmetric spaces}. Fourier Analysis. Birkh\"auser, Cham (2014), 149--186. MR3362019.

\bibitem[\textsc{Gro13}]{nadine}
\textsc{N.~Grosse}, \emph{The Yamabe equation on manifolds of bounded
	geometry}, Communications in Analysis and Geometry \textbf{21} (2013),
957--978.

\bibitem[\textsc{Guo19}]{wei}
\textsc{W.~Guodong}, \emph{Yamabe equation on some complete non-compact
	manifolds}, Pacific Journal of Mathematics \textbf{302} (2019), 717--739.

\bibitem[\textsc{Ham89}]{hamilton}
\textsc{R.~S. Hamilton}, \emph{Lectures on geometric flows}, unpublished
(1989).

\bibitem[HHM04]{hhm}
\textsc{T.~Hausel}, \textsc{E.~Hunsicker} and \textsc{R.~Mazzeo}, \emph{Hodge cohomology of gravitational instantons}. Duke Math. J. \textbf{122} (2004), no. 3, 485--548. MR2057017.

\bibitem[\textsc{IMS13}]{IMS}
\textsc{J. Isenberg}, \textsc{R. Mazzeo} and \textsc{N. Sesum}, \emph{Ricci flow on asymptotically conical surfaces with nontrivial topology} 
J. reine angew. Math., \textbf{676} (2013), 227-248.

\bibitem[\textsc{Kry96}]{krylov}
\textsc{N. V. Krylov}, \emph{Lectures on elliptic and parabolic equations in H\"older spaces} 
Graduate Studies in Mathematics, \textbf{12}. American Mathematical Society, Providence, RI, (1996).

\bibitem[\textsc{KrSa80}]{KS}
\textsc{N. V. Krylov} and \textsc{M. V. Safonov}, \emph{A property of the solutions of parabolic equations with measurable coefficients}, 
(Russian) Izv. Akad. Nauk SSSR Ser. Mat. \textbf{44} (1980), no. 1, 161--175, 239.

\bibitem[\textsc{LSU67}]{Lad}
\textsc{O.~A. Lady{\v{z}}enskaja, V.~A. Solonnikov} and \textsc{N.~N. Ural'ceva},
\emph{Linear and quasilinear equations of parabolic type}, Translated from
the Russian by S. Smith. Translations of Mathematical Monographs, Vol. \textbf{23},
American Mathematical Society, Providence, R.I., (1967).

\bibitem[LMP06]{lmp}
\textsc{E.~Leichtnam}, \textsc{R.~Mazzeo} and \textsc{P.~Piazza}, \emph{The index of Dirac operators on manifolds with fibered boundaries}. Bull. Belg. Math. Soc. Simon Stevin \textbf{13} (2006), no. 5, 845--855. MR2293212.

\bibitem[\textsc{LiCh14}]{macheng}
\textsc{L.~Ma} and \textsc{L.~Cheng}, \emph{Yamabe flow and Myers type theorem on complete manifolds}. J. Geom. Anal. \textbf{24} (2014), no. 1, 246--270. MR3145924.

\bibitem[Li19]{lima19}
\textsc{L.~Ma}, \emph{Yamabe flow and metrics of constant scalar curvature on a complete manifold}. Calc. Var. Partial Differential Equations \textbf{58} (2019), no. 1, Paper No. 30, 16 pp. MR3895773.

\bibitem[\textsc{Li21}]{lima}
\textsc{L.~Ma}, \emph{Global Yamabe flow on asymptotically flat manifolds}, preprint, arXiv:2102.02399 [math.DG] (2021).

\bibitem[\textsc{LCZ12}]{mcz}
\textsc{L.~Ma}, \textsc{L.~Cheng} and \textsc{A.~Zhu}, \emph{Extending
	Yamabe flow on complete Riemannian manifolds}, Bulletin des Sciences
Math\'{e}matiques \textbf{136} (2012), 882--891.


\bibitem[\textsc{MaMe98}]{mazzeo1998pseudodifferential}
\textsc{R.~Mazzeo} and \textsc{R.~B. Melrose}, \emph{Pseudodifferential
	operators on manifolds with fibred boundaries}, Mikio Sato: a great Japanese mathematician of the twentieth century. Asian J. Math. \textbf{4}, vol. 2 (1998): 833-866.
	
\bibitem[\textsc{Mel93}]{bookmelrose}
\textsc{R.~B.~Melrose}, \emph{The {A}tiyah-{P}atodi-{S}inger index theorem}, Research Notes in Mathematics, 4. A K Peters, Ltd., Wellesley, MA (1993). ISBN: 1-56881-002-4 MR1348401.

\bibitem[\textsc{O'N83}]{oneill}
\textsc{B.~O'Neill}, \emph{Semi-Riemannian geometry with applications to
	relativity}, Academic Press, 1983.

\bibitem[\textsc{Pic19}]{picard}
\textsc{S.~Picard}, \emph{Notes on H\"older Estimates for Parabolic PDE}
lecture notes, available online in http://people.math.harvard.edu/~spicard/notes-parabolicpde.pdf

\bibitem[\textsc{Sch19}]{schulzpaper}
\textsc{M.~B. Schulz}, \emph{Yamabe flow on non-compact manifolds with
	unbounded initial curvature}, The Journal of Geometric Analysis \textbf{30}
(2019), 4178--4192.


\bibitem[\textsc{Sch84}]{Schoen}
\textsc{R.~Schoen}, \emph{Conformal deformation of a {R}iemannian metric to
	constant scalar curvature}, J. Differential Geom. \textbf{20} (1984), no.~2, pp.
479--495.


\bibitem[\textsc{ScSt03}]{SS}
\textsc{H.~Schwetlick} and \textsc{M.~Struwe}, \emph{Convergence of the 
	{Y}amabe flow for ``large'' energies}, J. Reine Angew. Math. 
\textbf{562} (2003), pp. 59--100. \MR{2011332 (2004h:53097)}


\bibitem[\textsc{SSTa11}]{serratotapie}
\textsc{P.~Su\'{a}rez-Serrato} and \textsc{S.~Tapie}, \emph{Conformal entropy
	rigidity through Yamabe flow}, Mathematische Annalen \textbf{353} (2011),
333--357.

\bibitem[\textsc{Swa64}]{swan}
\textsc{R.~G. Swan}, \emph{Vector bundles and projective modules}, Matematika
\textbf{8} (1964), 29--44.

\bibitem[\textsc{TaVe21}]{VerTal}
\textsc{M.~Talebi} and \textsc{B.~Vertman}, \emph{Spectral geometry on
	manifolds with fibred boundary metrics ii: heat kernel asymptotics}, preprint, arXiv:2101.08844 [math.DG] (2021).


\bibitem[\textsc{Tru68}]{Trudinger}
\textsc{N.~S.~Trudinger}, \emph{Remarks concerning the conformal deformation of
	{R}iemannian structures on compact manifolds}, Ann. Scuola Norm. Sup. Pisa
(3) \textbf{22} (1968), pp. 265--274.

\bibitem[\textsc{Yam60}]{Yamabe}
\textsc{H.~Yamabe}, \emph{On a deformation of {R}iemannian structures on compact
	manifolds}, Osaka Math. J. \textbf{12} (1960), pp. 21--37.


\bibitem[\textsc{Ye94}]{ye}
\textsc{R.~Ye}, \emph{Global existence and convergence of Yamabe flow}, J.
Differential Geometry \textbf{39} (1994), 35--50.

\end{thebibliography}
\end{document}